\documentclass[11pt]{article}
\usepackage{amsmath} 
\usepackage{amsthm} 
\usepackage{amsfonts} 
\usepackage{amssymb}
\usepackage[all,cmtip]{xy}
\usepackage{graphicx}
\usepackage{color}
\newcommand{\R}{{\mathbb{R}}}
\newcommand{\Z}{{\mathbb{Z}}}
\newcommand{\ole}{\overline{e}}
\newcommand{\olf}{\overline{f}}
\newcommand{\ab}{\mathrm{ab}}

\DeclareMathOperator{\coker}{coker}

\DeclareMathOperator{\Crystal}{Crystal}
\DeclareMathOperator{\Div}{Div}
\DeclareMathOperator{\Image}{Image} 
\DeclareMathOperator{\Jac}{Jac}

\DeclareMathOperator{\pc}{pc}
\DeclareMathOperator{\Pic}{Pic}
\DeclareMathOperator{\Sk}{Sk}
\DeclareMathOperator{\source}{source}
\DeclareMathOperator{\sr}{sr}
\DeclareMathOperator{\target}{target}
\DeclareMathOperator{\valency}{valency}
\DeclareMathOperator{\Vor}{Vor}
\providecommand{\norm}[1]{\lVert#1\rVert} 
\providecommand{\abs}[1]{\lvert#1\rvert} 
\newtheorem{theorem}{Theorem}
\newtheorem{proposition}[theorem]{Proposition}

\newtheorem{lemma}[theorem]{Lemma}
\newtheorem{conjecture}[theorem]{Conjecture}
\newtheorem{lemma-definition}[theorem]{Lemma-Definition}
\theoremstyle{definition}
\newtheorem{definition}[theorem]{Definition}
\newtheorem{remark}[theorem]{Remark}
\newtheorem{example}[theorem]{Example}

\begin{document}


\begin{flushleft}
Tohoku Math.\ J., to appear
\end{flushleft}

\begin{center}
{\large\bfseries
VORONOI TILINGS HIDDEN IN CRYSTALS \\ ---THE CASE OF MAXIMAL ABELIAN COVERINGS---
}
\end{center}

\begin{center}
\textsc{Tadao ODA}\footnote{
Partly supported by JSPS Grant-in-Aid for Scientific Research (S-19104002)

\textit{$2010$ Mathematics Subject Classification.} Primary 52C22; Secondary 05C40, 14T05, 74E15, 
82B20, 82D25, 14M25.

}
\end{center}

\begin{abstract}
Consider a finite connected graph possibly with multiple edges and loops.
In discrete geometric analysis, Kotani and Sunada constructed the crystal associated to
the graph as a standard realization of the maximal abelian covering of the graph.
As an application of what the author showed in an earlier paper with Seshadri
as a by-product of Geometric Invariant Theory, he shows that 
the Voronoi tiling (also known as the Wigner-Seitz tiling) is hidden in the crystal, that is, 
the crystal does not intrude the interiors of the top-dimensional Voronoi cells.
The result turns out to be closely related to the tropical Abel-Jacobi map of 
the associated compact tropical curve.
\end{abstract}

\section*{Introduction}

Let $\Gamma$ be a finite connected graph possibly with multiple edges and loops.
With an orientation for each edge of $\Gamma$,
we can define 
\begin{align*}
C&:=C_1(\Gamma,\R)\supset C_1(\Gamma,\Z)=:\Lambda\quad\text{with the standard Euclidean metric}\\
H&:=H_1(\Gamma,\R)\supset H_1(\Gamma,\Z)=:H_{\Z}\\
 &\text{with the orthogonal projection }\pi\colon C\longrightarrow H.
\end{align*}

In discrete geometric analysis, Kotani-Sunada \cite{KotaniSunada1} constructed, 
for the maximal abelian covering $\Gamma^{\ab}\rightarrow\Gamma$ and a choice of 
a base vertex of $\Gamma$, a \emph{standard realization}
\[
\sr\colon\Gamma^{\ab}\longrightarrow\Crystal(\Gamma)\subset H,
\]
where $\Crystal(\Gamma)$ is a one-dimensional complex of line segments connecting
lattice points in $\pi(\Lambda)$. Moreover, $\sr$ is equivariant with respect to
the action of $H_{\Z}$ on $\Gamma^{\ab}$ and its translation action on $H$.

The standard realization of Kotani-Sunada collapses the bridges in the
graph $\Gamma$ so that
\[
(\Crystal(\Gamma))/H_1(\Gamma,\Z)\cong\overline{\Gamma},
\]
where $\overline{\Gamma}$ is the graph obtained by the collapse of the bridges.

In this paper (under the obvious assumption $\dim H\geq 2$) we show 
that $\Crystal(\Gamma)$ does not intrude the interiors of the
top-dimensional Voronoi cells (also known as Wigner-Seitz cells) in the Voronoi tiling
\[
\Vor(H,\xi_0+H_{\Z})
\]
for a suitable $\xi_0\in H$, that is, the Voronoi tiling $\Vor(H,\xi_0+H_{\Z})$ is ``hidden'' 
in the crystal $\Crystal(\Gamma)$.  

For example, a rhombic dodecahedral tiling is hidden in the diamond crystal
(cf.\ Example \ref{ex_diamond}), while
a truncated octahedral tiling is hidden in the $K_4$ crystal (cf.\ Example \ref{ex_K4}).

For the proof, we may assume without loss of generality that
the graph $\Gamma$ is bridgeless.
In fact, we can specify $\xi_0$ (hence the centers $\xi_0+H_{\Z}$ of the top-dimensional 
Voronoi cells) when $\Gamma$ is endowed with
a ``strongly connected'' orientation that is possible under the bridgelessness assumption
thanks to a result due to Robbins \cite{Robbins}.

A crucial role is played by Proposition \ref{prop_associatedVoronoi},
which describes the Voronoi cells in $\Vor(H,\xi_0+H_{\Z})$ arising out of graphs.
The proposition not only describes the facets of the Voronoi cells but gives a recipe 
for computing the vertices of the Voronoi cells. Hence it is an improvement of 
Oda-Seshadri \cite[Prop.\ 5.2, (1)]{OdaSeshadri}, which was a by-product of an attempt to apply 
the notion of stability in Geometric Invariant Theory to the compactifications of 
the generalized Jacobian varieties of algebraic curves with nodes.
(See Alexeev \cite{Alexeev} for later developments.)

More generally, Kotani-Sunada \cite{KotaniSunada1} considered the standard realization
of free abelian coverings $\widetilde{\Gamma}\rightarrow\Gamma$ that are not necessarily 
maximal abelian.
Denote
\[
L:=\Image\left(H_1(\widetilde{\Gamma},\Z)\rightarrow H_1(\Gamma,\Z)\right)\subset H_{\Z},
\]
which Sunada \cite{Sunada-book} calls the vanishing subgroup, 
and let $E'\subset H$ be the orthogonal complement in $H$ of the subspace spanned by $L$
\begin{align*}
E'&:=\{x\in H;\;(x,L)=0\}\subset H\subset C\\
  &\text{with the orthogonal projection }\pi'\colon C\longrightarrow E'.
\end{align*}
$H_{\Z}/L$ is the free abelian covering group of $\widetilde{\Gamma}\rightarrow\Gamma$.
Then the standard realization in this case is
\[
\widetilde{\sr}\colon\widetilde{\Gamma}\longrightarrow\pi'(\Crystal(\Gamma))\subset E'
\cong (H_{\Z}/L)\otimes_{\Z}\R.
\]
Even in this case, we expect that $\pi'(\Crystal(\Gamma))$ does not intrude the interiors
of the top-dimensional cells in a $\pi'(H_{\Z})$-periodic convex polyhedral tiling
of $E'$ (such as the Voronoi tiling
\[
\Vor(E',\xi_0'+\pi'(H_{\Z}))
\]
for a suitable choice of $\xi_0'\in E'$).

For instance, the Lonsdaleite crystal is the orthogonal projection onto the $3$-space of the
standard realization in the $5$-space of the maximal abelian covering.
A tiling by regular hexagonal cylinders, which is a Voronoi tiling
$\Vor(E',\xi'_0+\pi'(H_{\Z}))$ for a suitable $\xi_0\in E'$, turns out to be hidden in the 
Lonsdaleite crystal (cf.\ Example \ref{ex_lonsdaleite}). 

If $\Gamma$ is endowed with a strongly connected orientation and $\dim H\geq 2$, then
we can show the existence of a \emph{nondegenerate} $H_{\Z}$-periodic subdivision $\diamondsuit$ of 
$\Vor(H,\xi_0+H_{\Z})$ in the sense of Oda-Seshadri \cite[Prop.\ 7.6 and Thm.\ 7.7]{OdaSeshadri}
so that
\[
\Crystal(\Gamma)\subset\Sk^1(\diamondsuit),\quad\text{and}\quad
\pi(\Lambda)=\Sk^0(\diamondsuit).
\]
$\diamondsuit$ is obtained as one of the \emph{Namikawa tilings} 
(which we called Namikawa decompositions in \cite{OdaSeshadri}). The details will be
explained elsewhere (cf.\ \cite{OdaCutAndProject}).

The standard realization $\sr$ modulo $H_{\Z}$ turns out to be closely related to
the tropical Abel-Jacobi map
\[
\mu\colon\Gamma\longrightarrow\Jac(\Gamma):=H/H_{\Z}
\]
in the sense of Mikhalkin-Zharkov \cite{MikhalkinZharkov}, where
$\Gamma$ with weight $1$ for each edge is regarded as a compact tropical curve.
Our result implies that when the genus is greater than or equal to $2$, the image of 
the tropical Abel-Jacobi map $\mu$ lies in a translate of the tropical theta divisor in 
the tropical Jacobian variety $\Jac(\Gamma)$.

Thanks are due to Mathieu Dutour-Sikiri\'c for valuable information on Voronoi tilings.
Thanks are also due to Toshikazu Sunada who made available to the author the manuscript
for his forthcoming monograph \cite{Sunada-book}.
The author is grateful to Valery Alexeev for pointing out the relevance of tropical geometry.
Thanks are due to the referee for careful reading and valuable suggestions.

\section{Graphs and their associated Voronoi tilings}

A \emph{graph} $\Gamma$ in this paper is a \emph{finite connected} graph possibly with
multiple edges and loops (sometimes called a finite connected multigraph possibly with loops), 
that is, $\Gamma=\left(\{v_i\}_{i\in I}, \{e_j\}_{j\in J}\right)$ is a pair consisting of
a finite set $\{v_i\}_{i\in I}$ of vertices and a finite set $\{e_j\}_{j\in J}$ of edges.
Each edge $e_j$ joins two vertices $v_i$ and $v_{i'}$ for some $i,i'\in I$.
We call $e_j$ a loop when $i=i'$.

We assign and fix an \emph{orientation} of each edge, so that $e_j$
is regarded as an arrow from $v_i$ to $v_{i'}$, and we denote
\[
\source(e_j):=v_i,\quad\target(e_j):=v_{i'}.
\]
The orientation enables us to define the homology and cohomology groups of $\Gamma$:
\[
C_0(\Gamma,\Z):=\bigoplus_{i\in I}\Z v_i,\quad C_1(\Gamma,\Z):=\bigoplus_{j\in J}\Z e_j
\]
with the boundary and coboundary maps 
\[
\partial\colon C_1(\Gamma,\Z)\longrightarrow C_0(\Gamma,\Z),\qquad
\delta\colon C_0(\Gamma,\Z)\longrightarrow C_1(\Gamma,\Z)
\]
defined by
\begin{eqnarray*}
\partial(e_j)&=&\source(e_j)-\target(e_j),\\[1ex]
\delta(v_i)&=&\sum_{\substack{j\\ \source(e_j)=v_i}}e_j-\sum_{\substack{j\\ \target(e_j)=v_i}}e_j.
\end{eqnarray*}
We also consider their base extensions (e.g., $\partial_{\R}$ and $\delta_{\R}$) 
to the real numbers $\R$. We thus get
the first homology space $H_1(\Gamma,\R):=\ker(\partial_{\R})$ and its lattice 
$H_1(\Gamma,\Z):=\ker(\partial)$, 
while $H^1(\Gamma,\R):=\coker(\delta_{\R})$ and $H^1(\Gamma,\Z):=\coker(\delta)$,
although we do not use the the cohomology space and group below.
From now on, we denote $\partial_{\R}$ and $\delta_{\R}$ without the subscript $\R$.

We introduce the inner products on $C_0(\Gamma,\R)$ and $C_1(\Gamma,\R)$ by
\begin{eqnarray*}
[v_i,v_{i'}]&:=&\left\{\begin{array}{lcl}1&&i=i'\\ 0&&i\neq i' \end{array}\right.
\qquad\text{ for }v_i,v_{i'}\in C_0(\Gamma,\R),\\[2ex]
(e_j,e_{j'})&:=&\left\{\begin{array}{lcl}1&&j=j'\\ 0&&j\neq j'\end{array}\right.
\qquad\text{ for }e_j,e_{j'}\in C_1(\Gamma,\R).
\end{eqnarray*}

The inner product $(\;,\;)$ induces on $H_1(\Gamma,\R)$ a Euclidean metric.
Moreover, it is easy to see that $\partial$ and $\delta$ are adjoint to each other, i.e.,
\[
(\delta y,x)=[y,\partial x],\qquad\text{for all }y\in C_0(\Gamma,\R)\text{ and }x\in C_1(\Gamma,\R),
\]
hence we get an orthogonal decomposition
\[
C_1(\Gamma,\R)=H_1(\Gamma,\R)\oplus\delta C_0(\Gamma,\R)\qquad(\text{orthogonal}).
\]

For simplicity, we introduce the following notation:
\begin{itemize}
\item $C:=C_1(\Gamma,\R)\supset C_1(\Gamma,\Z)=:\Lambda$
\item $H:=H_1(\Gamma,\R)\supset H_1(\Gamma,\Z)=:H_{\Z}=\Lambda\cap H$
\item $\pi\colon C\rightarrow H$ the orthogonal projection
\end{itemize}

In this paper, we adopt the following convention of\ Bondy \cite{Bondy}:

\begin{definition} \label{def_walkpath}
Let $\Gamma$ be an oriented graph. 
 \begin{itemize}
 \item An edge $e$ is itself with the given orientation, while $-e$ is the edge with the orientation
 opposite to the given one. Thus
 \[
 \source(-e)=\target(e),\quad\target(-e)=\source(e).
 \]
 \item A \emph{walk} $w$ on $\Gamma$ is a finite sequence
 \[w\colon
 \xymatrix@C=4em{
 v_{i(0)}\ar[r]^{\varepsilon_{j(1)}e_{j(1)}}
  &v_{i(1)}\ar[r]^{\varepsilon_{j(2)}e_{j(2)}}
   &v_{i(2)}\ar[r]
    &
     & 
      & \\
  &
   &\cdots\ar[r]^{\varepsilon_{j(l)}e_{j(l)}}
    &v_{i(l)}\ar[r]
     &\cdots\ar[r]^(.45){\varepsilon_{j(m)}e_{j(m)}}
      &v_{i(m)}
 }
 \]
 with $\varepsilon_{j(l)}=\pm 1$ for all $1\leq l\leq m$ and
 \[
 \target(\varepsilon_{j(l)}e_{j(l)})=v_{i(l)}=\source(\varepsilon_{j(l+1)}e_{j(l+1)}),
 \quad\text{for all $l$ with }1\leq l\leq m-1.
 \]
 \item A \emph{path} is a walk with all the vertices $v_{i(l)}$'s distinct
 (hence the edges $e_{j(l)}$'s are distinct as well). 
 A \emph{directed path} is a path with $\varepsilon_{j(l)}=1$ for all $l$.
 \item A \emph{trail} is a walk with all the edges $e_{j(l)}$'s distinct.
 A \emph{directed trail} is a trail with $\varepsilon_{j(l)}=1$ for all $l$.
 \item A \emph{circuit} is a path coming back to the starting vertex, i.e., with $v_{i(0)}=v_{i(m)}$. 
 A \emph{directed circuit} is a directed path with $v_{i(0)}=v_{i(m)}$.
 \end{itemize}
\end{definition}

\begin{lemma} \label{lem_bridge}
For an edge $e$ of a \textup{(}connected\textup{)} graph $\Gamma$, the following are 
equivalent\textup{:}
 \begin{enumerate}
 \item[\textup{(i)}] $e$ is a bridge, i.e., its removal disconnects the graph $\Gamma$.
 
 \item[\textup{(ii)}] $e\in\delta C_0(\Gamma,\R)$.
 
 \item[\textup{(iii)}] $\pi(e)=0$.
 \end{enumerate}
\end{lemma}

\begin{remark}
A bridge is also called an isthmus, a cut edge, and a separating edge. 
\end{remark}

\begin{proof}
(ii) and (iii) are obviously equivalent.

(i) implies (ii). Indeed, suppose $e$ is a bridge. After its removal, denote by $I'\subset I$ 
the set of vertices of the connected component contining $\source(e)$. 
Then obviously $e=\delta(\sum_{i\in I'}v_i)$.

(iii) implies (i). Indeed, if $e$ is not a bridge, then there exists a circuit $\gamma$ containing
$e$. Then $(\gamma,\pi(e))=(\gamma,e)\neq 0$ with $\gamma$ regarded as an element
of $H_{\Z}=H_1(\Gamma,\Z)$ in an obvious way (cf.\ Definition \ref{def_lambda} below), 
hence $\pi(e)\neq 0$.
\end{proof}

\begin{proposition} \label{prop_bridgeless}
For a graph $\Gamma$, the following are equivalent\textup{:}
 \begin{enumerate}
 \item[$(1)$] $\Gamma$ is \emph{bridgeless}, i.e., $\Gamma$ remains connected after the removal
 of any edge. 
 
 \item[$(2)$] $\Gamma$ is \emph{$2$-edge connected}, i.e., any pair of vertices can be connected by
 two paths without common edges.
 
 \item[$(3)$] Any edge is contained in a circuit.
 
 \item[$(4)$] $\Gamma$ has a \emph{strongly connected orientation}, i.e., it can be so oriented that
 there exists a directed path from any vertex to another.
 \end{enumerate}
\end{proposition}

The equivalence of (1), (2) and (3) is well known and easy to prove.
(cf., e.g., Frank \cite[Prop.\ 2.10]{Frank}). The equivalence of (1) and (4) is due to
Robbins \cite{Robbins}. (See also Frank \cite[Cor.\ 2.13]{Frank}.)
Here is an apt description by Robbins \cite{Robbins} of the equivalence:
\begin{quote}
Let us suppose that week-day traffic in our city is not particularly heavy, so that
all streets are two-way, but that we wish to be able to repair any one street at a time and
still detour traffic around it so that any point in the city may be reached from any other point.
On week-ends no repairing is done, so that all streets are available, but due to heavy traffic
(perhaps it is a college town with a noted football team) we wish to make all streets one-way
and still be able to get from any point to any other without violating the law.
Then the theorem states that if our street-system is suitable for week-day traffic it is also
suitable for week-end traffic and conversely.
\end{quote}
As we see in Remark \ref{rem_bridgeless}, (3) below, we can describe all the possible strongly
connected orientations on a bridgeless graph by Kotani-Sunada \cite{KotaniSunada2}.

\begin{definition} \label{def_voronoi}
In a Euclidan space $E\cong\R^n$, consider a lattice $M\subset E$ and a coset $\Xi:=\xi_0+M\subset E$.
For each $\xi\in\Xi$, define its Voronoi cell (also known as Wigner-Seitz cell) centered at $\xi$ to be
\[
V(\xi):=\left\{x\in E;\;\; \norm{x-\xi}\leq\norm{x-\eta},\text{ for all }\eta\in\Xi\right\},
\]
which is a convex polytope. 
Different Voronoi cells do not intersect one another in their relative
interiors, and the union of the relative interiors of the members of
\[
\Vor(E,\xi_0+M):=\left\{V(\xi)\;(\xi\in\xi_0+M)\text{ and their faces}\right\}
\]
fills up the entire space $E$, giving rise to an $M$-periodic facet-to-facet convex polyhedral tiling 
called the \emph{Voronoi tiling} of $E$. More intuitively, the $M$-translates of the Voronoi cell
$V(\xi_0)$ centered at $\xi_0$ form a facet-to-facet tiling of $E$.
In particular, 
\[
E=V(\xi_0)+M
\]
so that $V(\xi_0)$ is a ``fundamental domain'' with respect to the translation action
of $M$ on $E$.
\end{definition}

\begin{definition} \mbox{} \label{def_eJVJDJ}
\begin{itemize}
\item For any subset $J'\subset J$, we denote
\[
e(J'):=\sum_{j\in J'}e_j\in C.
\]

\item In the Voronoi tiling $\Vor(C,\Lambda)$, the Voronoi cell centered at the origin is denoted
\begin{align*}
V_J&:=\sum_{j\in J}\left[-\frac{1}{2},\frac{1}{2}\right]e_j\\
   &=\left\{\sum_{j\in J}s_je_j;\;-\frac{1}{2}\leq s_j\leq\frac{1}{2}, 
   \text{ for all }j\in J\right\}\subset C.
\end{align*}

\item In the Voronoi tiling $\Vor(C,(e(J)/2)+\Lambda)$, the Voronoi cell centered at $e(J)/2$ is
\begin{align*}
D_J&:=\sum_{j\in J}[0,1]e_j\\
   &=\frac{e(J)}{2}+V_J\\
   &=\left\{\sum_{j\in J}s_je_j;\;0\leq s_j\leq 1\right\}\subset C.
\end{align*}
\end{itemize}
\end{definition}

\begin{remark}
For a subset $J_0\subset J$, we have
\[
D_J=e(J_0)+\sum_{j\in J_0}[0,1](-e_j)+\sum_{j\in J\setminus J_0}[0,1]e_j,
\]
so that the re-orientation of the edges $\{e_j;\; j\in J_0\}$ corresponds to 
a translation of the Voronoi tiling.
\end{remark}

\begin{definition} \label{def_graphtheoreticalcycle}
A graph-theoretical cycle (which we simply call a \emph{cycle} from now on) 
is an element of $H_{\Z}:=H_1(\Gamma,\Z)$ such that
$(\gamma,e_j)=0,\pm 1$ for all $j\in J$.
For any cycle $\gamma$, let us denote for simplicity:
\begin{align*}
\gamma^+&:=\{e_j;\; j\in J,(\gamma,e_j)=+1\}\\
\gamma^0&:=\{e_j;\; j\in J,(\gamma,e_j)=0\}\\
\gamma^-&:=\{e_j;\; j\in J,(\gamma,e_j)=-1\}.
\end{align*}
We denote their cardinalities by $\abs{\gamma^+}$, $\abs{\gamma^0}$ and $\abs{\gamma^-}$.
\end{definition}

For a subset $J'\subset J$, let us denote by $\{I,J'\}$ or $\{I,\{e_{j'}\}_{j'\in J'}\}$,
for simplicity, the subgraph 
of $\Gamma$ with the vertex set $\{v_i\}_{i\in I}$ and the edge set $\{e_{j'}\}_{j'\in J'}$.
(We follow the custom in graph theory so that a \emph{spanning} subgraph
is a subgraph with the same vertex set as that of the whole graph $\Gamma$.)

\begin{lemma-definition}[Elementary cycles, cf.\ {\cite[p.\ 21 and Lem.\ 4.6]{OdaSeshadri}}] 
\label{lem-def_elementarycycle}
A \emph{cycle} $\gamma$ is said to be an \emph{elementary cycle}, 
if the following equivalent conditions are satisfied\textup{:}
\begin{itemize}
\item $\gamma\neq 0$ and it cannot be written in the form $\gamma=\gamma_1+\gamma_2$
with nonzero cycles $\gamma_1,\gamma_2\in H_1(\Gamma,\Z)$ with
no common $e_j$ appearing in both $\gamma_1$ and $\gamma_2$ with nonzero coefficients.

\item There exists a spanning tree $T$ and $j\in J\setminus T$ such that either $\gamma$ or $-\gamma$
is of the form
\[
e_j+\sum_{t\in T}\varepsilon_te_t,\qquad\text{for}\quad\varepsilon_t=0,\pm1.
\]

\item $H_1(\{I,\gamma^+\cup\gamma^-\},\R)$ is one-dimensional $({}=\R\gamma)$.

\item $\{\pi(e_j);j\in\gamma^0\}$ spans a codimension one subspace of $H:=H_1(\Gamma,\R)$,
so that $\gamma^{\perp}=\sum_{j\in\gamma^0}\R\pi(e_j)$.
\end{itemize}
\end{lemma-definition}

The proof is easy.
Note that if $\gamma$ is an elementary cycle, then so is $-\gamma$.
Note also that $\gamma^+\cup\gamma^-$ (if appropriately ordered)
of an elementary cycle $\gamma$ is a \emph{circuit} in the sense of Definition \ref{def_walkpath}.

The following proposition plays a central role in this paper:

\begin{proposition}[Associated Voronoi tiling] 
\label{prop_associatedVoronoi}
\begin{enumerate}
\item[$(1)$] Let $H:=H_1(\Gamma,\R)$ with its lattice $H_{\Z}:=H_1(\Gamma,\Z)$ and
the orthogonal projection $\pi\colon C\rightarrow H$. Then $\pi(D_J)$ is
the Voronoi cell centered at $\pi(e(J)/2)$ in $\Vor(H,\pi(e(J)/2)+H_{\Z})$.
Hence the $H_{\Z}$-translates of $\pi(D_J)$ form a facet-to-facet tiling of $H$.

\item[$(2)$] We have
\[
\pi(D_J)=\left\{x\in H;\;(x,\gamma)\leq\abs{\gamma^+},\text{ for all elementary cycles }\gamma\right\},
\]
where the above system of defining inequalities are irredundant, i.e., the facets of $\pi(D_J)$
are exactly
\[
\left\{x\in H;\;(x,\gamma)=\abs{\gamma^+}\right\}\cap\pi(D_J)\quad\text{for elementary cycles $\gamma$}.
\]

\item[$(3)$] For each elementary cycle $\gamma$, the vertices of $\pi(D_J)$ lying on the facet
\[
\left\{x\in H;\;(x,\gamma)=\abs{\gamma^+}\right\}\cap\pi(D_J)
\]
are among
\[
\left\{\pi(e(Q));\;\gamma^+\subset Q\subset\gamma^+\cup\gamma^0\right\}.
\]
\end{enumerate}
\end{proposition}

\begin{remark} \mbox{} \label{rem_bridgeless}
$\{e(J');\;J'\subset J\}$ are the vertices of $D_J$. Hence 
\[
\pi(e(J'))\in\pi(D_J)\qquad\text{for all }J'\subset J.
\]
In particular, $0=\pi(e(\emptyset))\in\pi(D_J)$. 
By $(2)$, we see that $0$ is on the boundary of $\pi(D_J)$ if and only if
$\abs{\gamma^+}=0$ for an elementary cycle $\gamma$.

The above Proposition \ref{prop_associatedVoronoi} is an improvement of 
Oda-Seshadri \cite[Prop.\ 5.2, (1)]{OdaSeshadri}, since it not only describes the Voronoi cells
but (3) also gives a recipe for computing the vertices of the Voronoi cells. In fact,
we have more:
\begin{enumerate}
\item[(1)] We are in a situation of ``space tiling zonotopes'' dealt with by
Venkov \cite{Venkov}, McMullen \cite{McMullen1}, \cite{McMullen2}, \cite{McMullen3}, 
Erdahl-Ryshkov \cite{ErdahlRyshkov}, and Erdahl \cite{Erdahl}. Namely, 
\[
\pi(D_J)=\sum_{j\in J}[0,1]\pi(e_j)
\]
is a \emph{zonotope} whose $H_{\Z}$-translates give rise to a facet-to-facet tiling of $H$.
The family of hyperplanes
\[
\mathcal{H}\colon\qquad\{x\in H;\;(\pi(e_j),x)=m_j\},\qquad j\in J,\quad m_j\in\Z
\]
gives a \emph{dicing} in the sense of Erdahl-Ryshkov \cite{ErdahlRyshkov} and 
Erdahl \cite{Erdahl} with $H_{\Z}$ as the lattice 
and with the elementary cycles as the \emph{edge vectors}.
Moreover, the arrangement of the hyperplanes $\mathcal{H}$ coincides with the Delaunay tiling
of $H$ dual to $\Vor(H,H_{\Z})$.

\item[(2)] Define the \emph{support function} $s_D\colon H\rightarrow\R$ of $\pi(D_J)$ by
sending $u\in H$ to
\begin{align*}
s_D(u)&:=\max_{y\in\pi(D_J)}(u,y)=\max_{x\in D_J}(u,x)\\
      &=\max_{\substack{0\leq s_j\leq 1\\ j\in J}}\left(u,\sum_{j\in J}s_j e_j\right)\\
      &=\max_{\substack{0\leq s_j\leq 1\\ j\in J}}\sum_{j\in J}s_j(u,e_j)=\sum_{j\in J}(u,e_j)_+,
\end{align*}
since $(u,\pi(x))=(u,x)$ for all $u\in H$ and $x\in C$, where
\[
(u,e_j)_+:=
\left\{
\begin{array}{lcl}
(u,e_j)& &\text{if }(u,e_j)\geq 0\\
0           & &\text{otherwise}.
\end{array}
\right.
\]
Thus
\[
\pi(D_J)=\left\{x\in H;\;(u,x)\leq s_D(u),\text{ for all }u\in H\right\}.
\]
$s_D$ is \emph{strictly convex} and piecewise linear with respect to the \emph{normal fan}
$\Sigma$ of $\pi(D_J)$, which obviously is determined by the arrangement of the hyperplanes
\[
\mathcal{H}_0\colon\qquad
(\pi(e_j))^{\perp}=\left\{u\in H;\;(u,\pi(e_j))=0\right\},\quad\text{for $j\in J$}
\]
passing through the origin. Note that $\pi(e_j)\neq 0$ for all $j\in J$ 
if $\Gamma$ is bridgeless by Lemma \ref{lem_bridge}.
 \begin{itemize}
 \item The set $\Sigma(1)$ of one-dimensional cones in $\Sigma$ is in one-to-one correspondence
 with the set of facets of $\pi(D_J)$, hence with the set of elementary cycles so that
 \[
 \Sigma(1)=\left\{\R_{\geq 0}\gamma;\;\gamma\text{ elementary cycles}\right\}.
 \]
 For an elementary cycle $\gamma$, we have
 \[
 H_1(\{I,\gamma^+\cup\gamma^-\},\R)=\bigcap_{j\in\gamma^0}(\pi(e_j))^{\perp}
 =\R_{\geq 0}\gamma\cup\R_{\geq 0}(-\gamma).
 \]
 
 \item Denote $g:=\dim H$. Then the set $\Sigma(g)$ of top-dimensional cones in $\Sigma$ is in
 one-to-one correspondence with the set of vertices of $\pi(D_J)$.
 In fact, each cone $\sigma\in\Sigma(g)$ is the closure of a connected component of
 \[
 H\setminus\bigcup_{j\in J}(\pi(e_j))^{\perp}.
 \]
 Let
 \[
 Q:=\{j\in J;\;(\sigma,\pi(e_j))\subset\R_{\geq 0}\}.
 \]
 Then for $u\in\sigma$ we have
 \[
 s_D(u)=\sum_{j\in Q}(u,\pi(e_j))=(u,\pi(e(Q)))
 \]
 so that $\pi(e(Q))$ is the vertex of $\pi(D_J)$ corresponding to $\sigma$.
 
 If $\R_{\geq 0}\gamma_1,\ldots,\R_{\geq 0}\gamma_l$ are the one-dimensional faces of 
 the $g$-dimensional cone $\sigma$, and if
 $\Gamma$ is bridgeless so that $\pi(e_j)\neq 0$ for all $j\in J$ (cf.\ Lemma \ref{lem_bridge}),
 then
 \begin{align*}
 \sigma&=\R_{\geq 0}\gamma_1+\cdots+\R_{\geq 0}\gamma_l\\
 Q&=\bigcup_{1\leq k\leq l}\gamma_k^+=J\setminus\left(\bigcup_{1\leq k\leq l}\gamma_k^-\right).
 \end{align*}
 \end{itemize}
\item[(3)] In fact, we can say more thanks to Kotani-Sunada \cite[Theorem 1.3, (3)]{KotaniSunada2}:
Suppose $\Gamma$ is bridgeless in the sense of Proposition \ref{prop_bridgeless}.
Then $\Sigma(g)$ is in one-to-one correspondence with the set of strongly connected orientations
of $\Gamma$. 
Given a strongly connected orientation of $\Gamma$, the corresponding $\sigma\in\Sigma(g)$ is 
of the form
\[
\sigma=\R_{\geq 0}\gamma_1+\cdots+\R_{\geq 0}\gamma_l,
\]
where $\{\gamma_1,\ldots,\gamma_l\}$ is the set of elementary cycles $\gamma$ such that
\[
\gamma^{-}=\emptyset.
\]
(We can also show that $\pi(e(J))$ and $0$ are among the vertices of $\pi(D_J)$.)

Indeed, in connection with the \emph{homological directions of random walks} on $\Gamma$,
Kotani-Sunada \cite{KotaniSunada2} considers in $H$ the unit ball
\[
\overline{{\mathcal D}_0}:=\{x\in H;\;\norm{x}_1\leq 1\}
\]
with respect to the $l^1$-norm of $H$ induced by the $l^1$-norm
\[
\left\|\sum_{j\in J}c_je_j\right\|_1:=\sum_{j\in J}\abs{c_j}
\]
of $C$. Among others, Kotani and Sunada show that $\overline{{\mathcal D}_0}$
is a convex polytope with the set of vertices
\[
\left\{\frac{\gamma}{\norm{\gamma}_1};\;\gamma\text{ elementary cycles}\right\}
\]
and that the \emph{facets} of $\overline{{\mathcal D}_0}$ are in one-to-one correspondence
with the strongly connected orientations of $\Gamma$.
Since
\[
\norm{\gamma}_1=(\gamma,\gamma)
\]
for elementary cycles $\gamma$, we see that
\[
2\pi(V_J)=\{x\in H;\;(x,\gamma)\leq(\gamma,\gamma),\;\text{for all elementary cycles }\gamma\}
\]
is the polytope dual to $\overline{{\mathcal D}_0}$.
Thus the facets of $\overline{{\mathcal D}_0}$ are in one-to-one correspondence with
the vertices of $2\pi(V_J)$, hence with $\Sigma(g)$, since $\Sigma$ is the normal fan
of $2\pi(V_J)$ as well. $\Sigma$ is the fan consisting of the cones joining $0$ and
the proper faces of $\overline{{\mathcal D}_0}$.
\end{enumerate}
\end{remark}

\begin{proof}[Proof of Proposition $\ref{prop_associatedVoronoi}$]
(1) and (2) are consequences of \cite[Prop.\ 5.2, (1)]{OdaSeshadri} (in different notation) 
to the effect that
\[
\pi(V_J)=
\left\{x\in H;\;(x,\gamma)\leq\frac{(\gamma,\gamma)}{2},\text{ for all elementary cycles }\gamma\right\}
\]
with the defining inequalities above being irredundant so that $\pi(V_J)$ is the Voronoi cell 
centered at $0$ in $\Vor(H,H_{\Z})$.

Indeed, for each elementary cycle $\gamma$, we have
\[
(\gamma,\gamma)=\abs{\gamma^+}+\abs{\gamma^-}
\]
and
\[
(\pi(e(J)),\gamma)=(e(J),\gamma)=\abs{\gamma^+}-\abs{\gamma^-}
\]
(since $\gamma\in H$) so that $(x,\gamma)\leq(\gamma,\gamma)/2$ is equivalent to
\begin{align*}
\left(x+\frac{\pi(e(J))}{2},\gamma\right)&\leq\left(\frac{\pi(e(J))}{2},\gamma\right)+\frac{(\gamma,\gamma)}{2}\\
   &=\frac{\abs{\gamma^+}-\abs{\gamma^-}}{2}+\frac{\abs{\gamma^+}+\abs{\gamma^-}}{2}=\abs{\gamma^+}.
\end{align*}

Here is another direct proof of (2), which gives the proof of (3) as well:

If $\gamma$ is a cycle, then for $0\leq s_j\leq 1$ for all $j\in J$, we have
\begin{align*}
\left(\gamma,\pi\left(\sum_{j\in J}s_je_j\right)\right)&=\left(\gamma,\sum_{j\in J}s_je_j\right)\\
 &=\sum_{j\in J}s_j(\gamma,e_j)\leq\sum_{\substack{j\in J\\ (\gamma,e_j)>0}}(\gamma,e_j)=\abs{\gamma^+}.
\end{align*}
Hence 
\[
\pi(D_J)\subset\left\{x\in H;\;(x,\gamma)\leq\abs{\gamma^+},
\text{ for all cycle }\gamma\right\}.
\]

Suppose $H\ni\alpha\neq 0$ and $c\in\R$ determine a facet
$\{x\in H;\;(\alpha,x)=c\}\cap\pi(D_J)$ of $\pi(D_J)$. 
Let $\pi(e(Q_1)),\ldots,\pi(e(Q_r))$ for subsets $Q_1,\ldots,Q_r\subset J$ be the vertices of $\pi(D_J)$
on this facet, so that
\[
(\alpha,\pi(e(Q)))\leq c,\quad\text{ for all }Q\subset J
\]
with the equality holding if $Q=Q_1,\ldots,Q_r$.
Denote
\begin{align*}
\alpha^+&:=\{j\in J;\;(\alpha,e_j)>0\}\\
\alpha^0&:=\{j\in J;\;(\alpha,e_j)=0\}\\
\alpha^-&:=\{j\in J;\;(\alpha,e_j)<0\}.
\end{align*}
Hence $\alpha\in H_1(\{I,\alpha^+\cup\alpha^-\},\R)$ and is uniquely determined by
the condition
\[
(\alpha,\pi(e(Q)))=c,\quad\text{for }Q=Q_1,\ldots,Q_r.
\]
Hence $\dim H_(\{I,\alpha^+\cup\alpha^-\},\R)=1$ and $\alpha$ is a positive scalar multiple
of an elementary cycle $\gamma$ by Lemma-Definition \ref{lem-def_elementarycycle}.
Obviously, 
\[
(\gamma,\pi(e(Q)))\leq\abs{\gamma^+},\text{ for all }Q\subset J
\]
with the equality holding if and only if $\gamma^+\subset Q\subset\gamma^+\cup\gamma^0$.
In particular,
\[
\{Q_1,\ldots,Q_r\}\subset\{Q\subset J;\;\gamma^+\subset Q\subset\gamma^+\cup\gamma^0\}.
\]
\end{proof}

The following is a special case of a well-known result valid in the more general setting
where the inner product induced on $\Lambda$ is \emph{integral} and \emph{unimodular}.
Thanks are due to Mathieu Dutour-Sikiri\'{c} for pointing this out to the author.
In the following case of a graph, however, we have a recipe for constructing mutually dual $\Z$-bases as well.

\begin{lemma} \label{lem_duallattice}
$\pi(\Lambda)$ and $H_{\Z}:=H_1(\Gamma,\Z)=\Lambda\cap H$ are mutually dual lattices
in $H:=H_1(\Gamma,\R)$. Out of each spanning tree $T$ of $\Gamma$, 
mutually dual $\Z$-bases can be constructed.
\end{lemma}

\begin{proof}
Obviously,
\[
(\pi(\lambda),H_{\Z})=(\lambda,H_{\Z})\subset\Z,
\]
for all $\lambda\in\Lambda$ so that $\pi(\Lambda)$ is contained 
in the dual lattice of $H_{\Z}$.
Fix a spanning tree $T\subset J$. 
Then clearly
\[
C_1(\Gamma,\Z)=C_1(J\setminus T,\Z)+\delta C_0(\Gamma,\Z).
\]
Moreover for each $j\in J\setminus T$, there exists an elementary cycle $\gamma_j\in H_{\Z}$
with $\gamma_j-e_j\in C_1(T,\Z)$. We are done, 
since $\{\gamma_j\}_{j\in J\setminus T}$ is a $\Z$-basis of $H_{\Z}$ 
(cf., e.g., \cite[pp.\ 21--22]{OdaSeshadri}) and
\[
(\gamma_j,\pi(e_{j'}))=(\gamma_j,e_{j'})=
\left\{\begin{array}{lcl}
1& &j=j'\\
0& &j\neq j'.
\end{array}\right.
\]
\end{proof}

\section{Crystals}

We now briefly recall the standard realization and crystals due to Kotani-Sunada \cite{KotaniSunada1}.

Regard a (connected) graph $\Gamma$ as a one-dimensional cell complex in an obvious way, and 
fix a base vertex $v_0$. 
Denote by $\Omega(\Gamma)$ the universal covering space of $\Gamma$ regarded as 
consisting of the homotopy classes of ``curves'' on $\Gamma$ starting from $v_0$, 
while $\pi_1(\Gamma,v_0)$ is the fundamental group of $\Gamma$ regarded as consisting of
the homotopy classes of ``curves'' on $\Gamma$ starting from $v_0$ and ending at $v_0$. 
(We here use ``curve'' for customarily used ``path'' to avoid confusion with that defined in
Definition \ref{def_walkpath}.)

We have a canonical surjective map from $\Omega(\Gamma)$ to the maximal abelian covering
$\Gamma^{\ab}$ of $\Gamma$.
The concatenation action $\pi_1(\Gamma,v_0)\times\Omega(\Gamma)\rightarrow\Omega(\Gamma)$
induces the canonical action
\[
H_1(\Gamma,\Z)\times\Gamma^{\ab}\rightarrow\Gamma^{\ab}.
\]

A walk $w$ on $\Gamma$ from $v_0$ gives rise to a point on $\Omega(\Gamma)$ (and on $\Gamma^{\ab}$)
as well as a ``curve'' on $\Omega(\Gamma)$ (and on $\Gamma^{\ab}$) ending at the point.
For simplicity, we denote also by $w$ the point as well as the ``curve''.
($\Omega(\Gamma)$ and $\Gamma^{\ab}$ are also regarded as ``infinite graphs'' with these points
as vertices.)

\begin{definition} \label{def_lambda}
For a walk $w$ on $\Gamma$ of the form
\[
\xymatrix@C=4em{
v_{i(0)}\ar[r]^{\varepsilon_{j(1)}e_{j(1)}}
 &v_{i(1)}\ar[r]^{\varepsilon_{j(2)}e_{j(2)}}
  &v_{i(2)}\ar[r]
   &
    & 
     & \\
 &
  &\cdots\ar[r]^{\varepsilon_{j(l)}e_{j(l)}}
   &v_{i(l)}\ar[r]
    &\cdots\ar[r]^(.45){\varepsilon_{j(m)}e_{j(m)}}
     &v_{i(m)}
}
\]
we define the element $\lambda(w)\in\Lambda=C_1(\Gamma,\Z)$ by
\[
\lambda(w):=
\left\{
\begin{array}{lcl}
0 & & \text{if }m=0\\
\sum_{1\leq l\leq m}\varepsilon_{j(l)}e_{j(l)}& & \text{if }m>0.
\end{array}
\right.
\]
\end{definition}

When a walk $w$ on $\Gamma$ starts from $v_0=v_{i(0)}$ in the above notation, 
we denote by $w_l$ the part of $w$ from $v_0$ to $v_{i(l)}$ for $l=0,\ldots,m$.
Then we have a \emph{$\Lambda$-polygonal curve} in $C:=C_1(\Gamma,\R)$
\[
\pc(w):=\bigcup_{1\leq l\leq m}\left[\lambda(w_{j(l-1)}),\lambda(w_{j(l)})\right]\subset C,
\]
where $[\lambda,\lambda']$ for $\lambda,\lambda'\in\Lambda$ is the 
\emph{$\Lambda$-line segment} in $C$ joining $\lambda$ and $\lambda'$.

Applying the orthogonal projection $\pi\colon C\rightarrow H:=H_1(\Gamma,\R)$, we have a
map from $\Omega(\Gamma)$ to the set of \emph{$\pi(\Lambda)$-polygonal curves}
$\pi(\pc(w))$ in $H$, which induces the \emph{standard realization}
\[
\sr\colon\Gamma^{\ab}\longrightarrow\Crystal(\Gamma):=
\bigcup_{w\text{ walks from $v_0$}}\pi(\pc(w))\subset H
\]
of Kotani-Sunada \cite{KotaniSunada1} (see also Sunada \cite{Sunada}, \cite{Sunada-book}).

The standard realization $\sr$ is equivariant with respect to the concatenation action of 
$H_{\Z}:=H_1(\Gamma,\Z)$ on $\Gamma^{\ab}$ and the translation action of $H_{\Z}$ on $H$.
The crystal $\Crystal(\Gamma)$ is a one-dimensional complex in $H$ of \emph{$\pi(\Lambda)$-line segments}
joining lattice points in $\pi(\Lambda)$ and is $H_{\Z}$-periodic, i.e., invariant
under the translation action of the sublattice $H_{\Z}$.

The standard realization $\sr$ collapses the bridges in the graph $\Gamma$, and
\[
\Crystal(\Gamma)/H_1(\Gamma,\Z)\cong\overline{\Gamma},
\]
where $\overline{\Gamma}$ is the graph obtained by the collapse of the bridges in $\Gamma$.

Thus to show that $\Vor(H,\xi_0+H_{\Z})$ for some $\xi_0\in H$ is ``hidden'' in $\Crystal(\Gamma)$,
we may assume without loss of generality that the graph $\Gamma$ is bridgeless.

More generally, Kotani-Sunada \cite{KotaniSunada1} considered also the crystal for 
free abelian coverings $\widetilde{\Gamma}$ of $\Gamma$ that are not necessarily maximal abelian.

Let $L\subset H_{\Z}$ be the image of $H_1(\widetilde{\Gamma},\Z)\rightarrow H_1(\Gamma,\Z)$
so that $L$ is a subgroup such that $\Gamma^{\ab}/L=\widetilde{\Gamma}$ and that
$H_{\Z}/L$ is the free abelian covering group for
$\widetilde{\Gamma}\rightarrow\Gamma$.

Denote
\begin{align*}
E'&:=\{x\in H;\;(x,L)=0\}\subset H\\
  &\cong\left(H_{\Z}/L\right)\otimes_{\Z}\R \\
 &\text{with }\pi'\colon C\rightarrow E'\text{ the orthogonal projection}.
\end{align*}
$E'$ has lattices 
\[
\Lambda\cap E'\subset\pi'(H_{Z})\subset\pi'(\Lambda).
\]
$\Lambda\cap E'$ and $\pi'(\Lambda)$ are \emph{mutually dual} lattices, since the inner product on $C$
is integral and unimodual with respect to the orthonormal lattice $\Lambda$.
The standard realization of Kotani-Sunada in this more general setting is
\[
\widetilde{\sr}\colon \widetilde{\Gamma}\longrightarrow\pi'(\Crystal(\Gamma))\subset 
E'\cong\left(H_{\Z}/L\right)\otimes_{\Z}\R.
\]
$\pi'(\Crystal(\Gamma))$ is a one-dimensional complex of \emph{$\pi'(\Lambda)$-line segments} 
joining lattice points in $\pi'(\Lambda)$, and is $\pi'(H_{\Z})$-periodic, i.e., invariant 
under the translation action of the sublattice $\pi'(H_{\Z})$.

We have a commutative diagram
\[
\xymatrix@C=5em@R=5ex{
\Gamma^{\ab}\ar[r]^{\sr}\ar[d]& \Crystal(\Gamma)\ar[d]^{\pi'} \\
\widetilde{\Gamma}=\Gamma^{\ab}/L\ar[r]^{\widetilde{\sr}} & \pi'(\Crystal(\Gamma)).
}
\]
$\widetilde{\sr}$ induces 
\[
\widetilde{\Gamma}/(H_{\Z}/L)\longrightarrow\pi'(\Crystal(\Gamma))/\pi'(H_{\Z}),
\]
which need not be bijective, even if $\Gamma$ is bridgeless so that
\[
\Gamma^{\ab}/H_{\Z}\longrightarrow\Crystal(\Gamma)/H_{\Z}
\]
is bijective. 

\section{Main theorem}

We are now ready to prove the following theorem which was conjectured 
at the 2010 Annual Meeting of the Japan Society of Industrial and Applied Mathematics
held at Meiji University in Tokyo on September 6, 2010:

\begin{theorem} \label{thm_maintheorem}
Let $\Gamma=\{I, J\}$ be a bridgeless graph with $\dim H_1(\Gamma,\R)\geq 2$.
Then after a strongly connected re-orientation and a change of the base vertex $v_0$ if necessary, 
the crystal obtained as the standard realization of the maximal abelian covering $\Gamma^{\ab}$
does not intrude the interiors of the top-dimensional Voronoi cells in
\[
\Vor\left(H_1(\Gamma,\R),\pi\left(\frac{\sum_{j\in J}e_j}{2}\right)+H_1(\Gamma,\Z)\right),
\]
that is, for some $r<\dim H_1(\Gamma,\R)$ we have\textup{:}
\[
\Crystal(\Gamma)\subset
\Sk^r\left(\Vor\left(H_1(\Gamma,\R),\pi\left(\frac{\sum_{j\in J}e_j}{2}\right)+H_1(\Gamma,\Z)\right)\right),
\]
where $\Sk^r$ denotes the $r$-skeleton. Thus a Voronoi tiling is ``hidden'' in the crystal $\Crystal(\Gamma)$.
\end{theorem}

\begin{lemma} \label{lem_strongconnectivity}
Let $\Gamma$ be a strongly connected graph with $\dim H_1(\Gamma,\R)\geq 2$.
Then $\Gamma$ has one of the following \textup{(}not necessarily spanning\textup{)} 
subgraphs  $\Gamma_1$ and $\Gamma_2$\textup{:}
 \begin{itemize}
 \item $\Gamma_1$ consists of two directed circuits $\gamma_1$ and $\gamma_2$ meeting only at a 
 vertex $v^{\ast}$.
 
 \item $\Gamma_2$ consists of three directed paths $p_1$, $p_2$, $p_3$ between two distinct
 vertices $v^{\ast}$ and $v^{\ast\ast}$ with $p_1$, $p_2$, $p_3$ disjoint except at the
 end vertices $v^{\ast}$ and $v^{\ast\ast}$, where $p_1$ is a directed path from $v^{\ast\ast}$
 to $v^{\ast}$, while $p_2$ and $p_3$ are paths from $v^{\ast}$ to $v^{\ast\ast}$.
 \end{itemize}
\[
\xymatrix@C=2em{
\gamma_1 \ar@/_2pc/[r]& v^{\ast} \ar@{-}@/_2pc/[l]  \ar@{-}@/^2pc/[r] & \gamma_2 \ar@/^2pc/[l]\\
 &\Gamma_1 &
 }
 \qquad\qquad
 \xymatrix{
 v^{\ast}\ar[rr]|{p_2} \ar@/_2pc/[rr]|{p_3} & & v^{\ast\ast}\ar@/_2pc/[ll]|{p_1}\\
 &\Gamma_2&
 }
\]
\end{lemma}

\begin{proof} \mbox{}
(i) If $\Gamma$ has only one vertex, then $\Gamma$ consists of at least two loops so that it has
a subgraph $\Gamma_1$.

(ii) Suppose $\Gamma$ has only two vertices.
If there are two loops at a vertex, then they form a subgraph $\Gamma_1$.
If there exists only one loop at a vertex, then there exist two edges in opposite directions connecting
the vertex with the other by the strong-connectivity assumption, 
and we obviously have a subgraph $\Gamma_1$.
If $\Gamma$ has no loops, then there exist at least three edges connecting 
the two vertices with at least two edges in opposite directions by our assumptions.
Thus $\Gamma$ has a subgraph $\Gamma_2$.

(iii) Suppose now that $\Gamma$ has at least three vertices. By assumption, we can choose a directed
circuit $\gamma$.

(iii-a) If $\gamma$ is Hamiltonian (i.e., passes through all the vertices of $\Gamma$), then 
since $\dim H_1(\Gamma,\R)\geq 2$, there exists either a loop or an edge $e$ connecting two distinct
vertices. Obviously, we have a subgraph $\Gamma_1$ in the former case and $\Gamma_2$ in the
latter case.

(iii-b) If $\gamma$ is not Hamiltonian, take a vertex $v$ not on $\gamma$.
By assumption, there exists a directed path $p'$ from $v$ to a vertex $v'$ on $\gamma$
with $p'$ and $\gamma$ disjoint except at $v'$. There also exists a directed path $p''$ to $v$ from
a vertex $v''$ on $\gamma$ with $p''$ and $\gamma$ disjoint except at $v''$.
If $v'\neq v''$, then we may assume that $p'$ and $p''$ are disjoint except at $v$. 
Indeed, otherwise, there certainly exists a vertex $v'''$ that is closest to $v'$ on $p'$ 
as well as to $v''$ on $p''$. We then replace $v$ by $v'''$. 
In this case, $\Gamma$ has a subgraph $\Gamma_2$.
If $v'=v''$, then we may again assume that $p'$ and $p''$ are disjoint except
at $v$ and $v'$. In this case $\Gamma$ has a subgraph $\Gamma_1$.
\[
\xymatrix@C=3em@R=1ex{
&& \\
                  & v' \ar@{-}@/^2pc/[rd]            & \\
v\ar@/^/[ru]|{p'} &                                  &\gamma \ar@/^2pc/[ld] \\
                  & v''\ar@/^/[lu]|{p''} \ar@/^/[uu] & \\
}
\qquad\qquad
\xymatrix@R=1ex@C=2em{
& & \\
& & \\
v \ar@/^1pc/[r]|{p'} & v' \ar@/^1pc/[l]|{p''} \ar@{-}@/^2pc/[r] & \gamma \ar@/^2pc/[l]\\
& & \\
}
\]
\end{proof}

To prove Main Theorem \ref{thm_maintheorem}, let us endow $\Gamma$ with a strongly connected
orientation, which is possible by Proposition \ref{prop_bridgeless}.
Since $\dim H_1(\Gamma,\R)\geq 2$, Lemma \ref{lem_strongconnectivity} guarantees the existence
of a subgraph $\Gamma_1$ or $\Gamma_2$. In either case, let $v_0:=v^{\ast}$ be the base vertex
for the standard realization of the crystal. 
Recall our definition in Definition \ref{def_lambda} of $\lambda(w)$ for a walk on $\Gamma$.

\begin{lemma} \label{lem_directedpath}
Under the choice of the strongly connected orientation and the base vertex $v_0$, for any walk
$w$ on $\Gamma$ from $v_0$ there exists a directed path $p$ from $v_0$ to the same end vertex
as that of $w$ such that $\lambda(w)-\lambda(p)\in H_1(\Gamma,\Z)$. 
Moreover, for directed paths $p$ and $p'$ from $v_0$ to $v$, we have
$\lambda(p)-\lambda(p')\in H_1(\Gamma,\Z)$.
\end{lemma}
 
\begin{proof} \mbox{}
The second assertion is obvious.

As for the first, suppose
the walk $w$ starts out with a path $p_1$ from $v_0$ to a vertex $v$ followed by
an edge $-e$ from $v$ to $v'$:
\[
w=p_1+(-e)+q,
\]
where $q$ is the part of the walk $w$ from $v'$ onward.
Since $e$ is an edge from $v'$ to $v$, there exists a
directed path $p'$ from $v$ to $v'$ by the strong connectivity. 

If $p'$ and $p_1$ are disjoint except at the vertex $v$, then replace the beginning of the walk $w$ by
$p_1+p'$ and consider
\[
w':=p_1+p'+q.
\]
Obviously, $\lambda(w)-\lambda(w')\in H_1(\Gamma,\Z)$.

If $p'$ and $p_1$ have common edges, there certainly exists a subpath $p''$ of $p'$ 
(the ending part of the directed path $p'$)
from a vertex $v''$ on $p_1$ to $v'$ such that $p''$ and $p_1$ are disjoint except at $v''$.
Denote by $p_1'$ (resp.\ $p_1''$) the subpath of $p_1$ from $v_0$ to $v''$ 
(resp.\ from $v''$ to $v$). 
\[
\xymatrix@C=2em{
v_0 \ar[rr]|{p_1'}& &\ar[rr]|{p_1''}\ar[rrrd]|{p''} v'' && v\ar[rd]^{-e} & &\\
    &&    & &   & v'\ar@{--}[r]&
}
\]
Then replace the beginning of the walk $w$ by $p_1'+p''$ and consider
\[
w'':=p_1'+p''+q.
\]
Obviously, $\lambda(w)-\lambda(w'')\in H_1(\Gamma,\Z)$.

The rest of the proof follows by induction on the number of edges in $w$ appearing
with the minus sign.
\end{proof}

\begin{proof}[Proof of Main Theorem $\ref{thm_maintheorem}$] \mbox{}
By Proposition \ref{prop_associatedVoronoi}, (1), the $H_{\Z}$-translates of $\pi(D_J)$
form a facet-to-facet tiling of $H$ so that
\[
H=\pi(D_J)+H_{\Z},
\]
that is, $\pi(D_J)$ is a ``fundamental domain'' with respect to the translation action of
$H_{\Z}$ on $H$.

By Lemma \ref{lem_directedpath}, it suffices to consider
only directed paths on $\Gamma$ from $v_0$.
Since the images by $\pc$ in $C$ of the directed paths from $v_0$ are contained in $D_J$, 
their projections under $\pi$ are contained in $\pi(D_J)$, whose facets are of the form
\[
\{x\in H;\;(\gamma,x)=\abs{\gamma^+}\}\cap\pi(D_J)\qquad\text{for an elementary cycle $\gamma$}.
\]
To show that the images under $\pi$ of the directed paths from $v_0$ are on the boundary of
$\pi(D_J)$,  it suffices to show the following:

For any edge $e$ and an appropriate choice of a directed path $p$ from $v_0$ to $v:=\source(e)$,
\[
p\colon\quad 
\xymatrix{
v_0\ar[r]^{e_1} & v_1\ar[r]^{e_2} & v_2\ar[r] &\cdots\ar[r]^(.4){e_m} & v_m=v,
}
\]
there exists an elementary cycle $\gamma$ such that
\begin{align*}
(\gamma,\lambda(p))&=\abs{\gamma^+}\\
(\gamma,e)&=0,
\end{align*}
where
\[
\lambda(p):=e_1+e_2+\cdots+e_m
\]
as in Definition \ref{def_lambda}.
By Lemma \ref{lem_strongconnectivity}, we have two cases:

(Case 1) $\Gamma$ contains a subgraph $\Gamma_1$ with two directed circuits $\gamma_1$ and $\gamma_2$ 
meeting only at the common vertex $v_0$.

If $v:=\source(e)$ is on $\gamma_1$, choose the directed path $p$ from $v_0$ to $v$ to be the
directed path along $\gamma_1$ from $v_0$ to $v$. Let $\gamma:=-\lambda(\gamma_2)$. Then we are done, since
\begin{align*}
(\gamma,\lambda(p))&=0=\abs{\gamma^+}\\
(\gamma,e)&=0.
\end{align*}
\[
\xymatrix@C=3em@R=3em{
 &v \ar[l]^e & & \\
 &\gamma_1 \ar@{--}@/^/[u] \ar@{-->}@/_2pc/[r] & v_0 \ar@/_1pc/[lu]|{p} \ar@{-}@/^2pc/[r]
                                                & \gamma_2 \ar@/^2pc/[l] \\
 & & & \\
}
\qquad\qquad 
\xymatrix@C=3em@R=3em{
 &v\ar[l]^e &v'\ar@/_1pc/[l]|{p''} & & \\
 & &\gamma_1\ar@{--}@/^/[u] \ar@{-->}@/_2pc/[r] &v_0 \ar@/_1pc/[lu]|{p'} \ar@{-}@/^2pc/[r] 
                                                 &\gamma_2 \ar@/^2pc/[l]
}
\]

If $v$ is on $\gamma_2$, we are done in a similar manner.

Suppose $v$ is neither on $\gamma_1$ nor on $\gamma_2$. 
Without loss of generality, we may assume the existence of a directed path $p''$ to $v$ 
from a vertex $v'$ on $\gamma_1$ such that $p''$ and $\Gamma_1$ are disjoint except at $v'$.
Denote by $p'$ the directed path along $\gamma_1$ from $v_0$ to $v'$, and let $p:=p'+p''$ be the direct
path $p'$ followed by $p''$. Then $\gamma:=-\lambda(\gamma_2)$ again satisfies
\begin{align*}
(\gamma,\lambda(p))&=0=\abs{\gamma^+}\\
(\gamma,e)&=0,
\end{align*}
and we are done.

(Case 2) $\Gamma$ contains a subgraph $\Gamma_2$ consisting of three paths 
$p_1$, $p_2$, $p_3$, disjoint except at the end vertices, 
between $v_0$ and $v^{\ast\ast}$ with $p_1$ directed from 
$v^{\ast\ast}$ to $v_0$, while $p_2$ and $p_3$ are directed from $v_0$ to $v^{\ast\ast}$.
\[
\xymatrix@C=2em{
\Gamma_2\colon& & v_0 \ar[rr]|{p_2} \ar@/_2pc/[rr]|{p_3}  && v^{\ast\ast} \ar@/_2pc/[ll]|{p_1}
}
\]

For an edge $e$ with $v:=\source(e)$, we have four cases to consider:
\begin{enumerate}
\item[(i)] $v$ is on $p_1$. Let $p'$ be the directed path along $p_1$ from $v^{\ast\ast}$ to $v$,
and let $p:=p_2+p'$ be the directed path $p_2$ followed by $p'$. 
Then the elementary cycle $\gamma:=\lambda(p_2)-\lambda(p_3)$
satisfies
\begin{align*}
(\gamma,\lambda(p))&=(\text{length of $p_2$})=\abs{\gamma^+}\\
(\gamma,e)&=0,
\end{align*}
and we are done.
\[
\xymatrix@C=2em@R=2em{
 & & \\
 & v \ar[lu]^e \ar@{-->}@/_1pc/[dl]& \\
v_0\ar[rr]|{p_2}\ar@/_2pc/[rr]|{p_3} & & v^{\ast\ast} \ar@/_1pc/[lu]|{p'}\\
 &\text{(i)} &
}
\qquad\qquad
\xymatrix@C=2em@R=2em{
v_0 \ar[rr]|{p_2} \ar@/_1pc/[rd]|{p} & & v^{\ast\ast} \ar@/_2pc/[ll]|{p_1} \\
    &v \ar[dr]^e \ar@{-->}@/_1pc/[ru] & \\
    &\text{(ii)} &
}
\]

\item[(ii)] $v$ is on $p_3$. Let $p$ be the directed path along $p_3$ from $v_0$ to $v$.
Then the elementary cycle $\gamma:=-\lambda(p_1)-\lambda(p_2)$ satisfies
\begin{align*}
(\gamma,\lambda(p))&=0=\abs{\gamma^+}\\
(\gamma,e)&=0,
\end{align*}
and we are done.

\item[(iii)] $v$ is on $p_2$. The proof is similar to that for (ii).

\item[(iv)] $v$ is not on $p_1$, $p_2$, $p_3$.
There certainly exists a directed path $p'$ from a vertex $v'$ of $\Gamma_2$ to $v$
such that $p'$ and $\Gamma_2$ are disjoint except at $v'$. We have three cases to consider:
 \begin{enumerate}
 \item[(iv-a)] $v'$ is on $p_1$. Let $p_1'$ be the directed path along $p_1$ from
 $v^{\ast\ast}$ to $v'$ and let the directed path $p:=p_2+p_1'+p'$ from $v_0$ to $v$
 be $p_2$ followed by $p_1'$ and then by $p'$. 
 Then the elementary cycle $\gamma:=\lambda(p_2)-\lambda(p_3)$ satisfies
 \begin{align*}
 (\gamma,\lambda(p))&=(\text{length of $p_2$})=\abs{\gamma^+}\\
 (\gamma,e)&=0,
 \end{align*}
 and we are done.
 \[
 \xymatrix@C=2em{
  & & v\ar[dr]^e & \\
  &v'\ar@/^/[ur]|{p'} \ar@{-->}@/_/[dl] & & \\
 v_0 \ar[rr]|{p_2} \ar@/_2pc/[rr]|{p_3} & &v^{\ast\ast} \ar@/_/[ul]|{p_1'} & \\
  & & \text{(iv-a)} &
 }
 \qquad\qquad
 \xymatrix@C=2em{
 v_0\ar[rr]|{p_2} \ar@/_/[rd]|{p_3'} &                    & v^{\ast\ast}\ar@/_2pc/[ll]|{p_1} & \\
                                     & v'\ar@/_/[rd]|{p'} \ar@{-->}@/_/[ru] &             & \\
                                     &                    & v \ar[ru]^e                      & \\
                                     & \text{(iv-b)} & &
 }
 \]
 
 \item[(iv-b)] $v'$ is on $p_3$. Let $p_3'$ be the directed path along $p_3$ from $v_0$ to $v'$
 and let the directed path $p:=p_3'+p'$ from $v_0$ to $v$ be $p_3'$ followed by $p'$.
 Then the elementary cycle $\gamma:=-\lambda(p_1)-\lambda(p_2)$ satisfies
 \begin{align*}
 (\gamma,\lambda(p))&=0=\abs{\gamma^+}\\
 (\gamma,e)&=0,
 \end{align*}
 and we are done.
 
 \item[(iv-c)] $v'$ is on $p_2$. The proof is similar to that for (iv-b).
 \end{enumerate}
\end{enumerate}
\end{proof}

\begin{remark} \label{rem_directedtrail}
As the proof shows, $(\Crystal(\Gamma))\cap\pi(D_J)$ consists of the $\pi(\Lambda)$-polygonal
curves arising out of directed trails from $v_0$.
\end{remark}

\begin{remark} \label{rem_nondegenerateNamikawa}
If $\Gamma$ is endowed with a strongly connected orientation and $\dim H\geq 2$, 
then we can show the existence of a \emph{nondegenerate} $H_{\Z}$-periodic subdivision $\diamondsuit$ of
$\Vor(H,\pi(e(J)/2)+H_{\Z})$ in the sense of 
Oda-Seshadri \cite[Prop.\ 7.6 and Thm.\ 7.7]{OdaSeshadri} so that
\begin{align*}
\Crystal(\Gamma)&\subset\Sk^1(\diamondsuit)\\
\pi(\Lambda)&=\Sk^0(\diamondsuit).
\end{align*}
$\diamondsuit$ is obtained as one of the Namikawa tilings (which we called Namikawa decompositions
in \cite{OdaSeshadri}). The details are given elsewhere (cf.\ \cite{OdaCutAndProject}).
\end{remark}

Let $\widetilde{\Gamma}\rightarrow\Gamma$ be a non-maximal abelian covering
with the vanishing subgroup
\[
L:=\Image\left(H_1(\widetilde{\Gamma},\Z)\longrightarrow H_1(\Gamma,\Z)\right)
\]
and the standard realization
\[
\widetilde{\sr}\colon\widetilde{\Gamma}\longrightarrow\pi'(\Crystal(\Gamma))
\subset E'\cong(H_{\Z}/L)\otimes_{\Z}\R.
\]

\begin{conjecture} \label{conjecture}
For the non-maximal abelian covering $\widetilde{\Gamma}\rightarrow\Gamma$,
the crystal $\pi'(\Crystal(\Gamma))$ does not intrude the interiors of the
top-dimensional tiles in a $\pi'(H_{\Z})$-periodic convex polyhedral tiling
of $E'$.
\end{conjecture}

For instance, the Lonsdaleite crystal is the orthogonal projection onto the $3$-space of the
standard realization in the $5$-space of the maximal abelian covering.
A tiling by regular hexagonal cylinders, which is a Voronoi tiling
$\Vor(E',\xi'_0+\pi'(H_{\Z}))$ for a $\xi'_0\in E'$, turns out to be hidden in the Lonsdaleite 
crystal (cf.\ Example \ref{ex_lonsdaleite}). 

Depending on $L$, however, convex polyhedral tilings other than Voronoi tilings may be needed.
It would be of interest to characterize $L$ for which $\Vor(E',\xi'_0+\pi'(H_{\Z}))$
for some $\xi'_0$ is the $\pi'(H_{\Z})$-periodic convex polyhedral tiling in question.

\begin{remark} \label{rem_tropical}
As in Mikhalkin-Zharkov \cite{MikhalkinZharkov}, our finite connected bridgeless graph $\Gamma$
can be regarded as a \emph{compact tropical curve} with the metric of length $1$ for each edge.
The real torus $\Jac(\Gamma):=H/J_{\Z}$ turns out to be its \emph{tropical Jacobian variety}
with $g:=\dim H$ being the \emph{genus} of $\Gamma$.

Suppose $g\geq 2$ and endow $\Gamma$ with a strongly connected orientation.
Denote by $\varpi\colon H\rightarrow\Jac(\Gamma)$ the canonical projection. Then the standard realization
$\sr\colon\Gamma^{\ab}\rightarrow H$ induces modulo $H_{\Z}$ the \emph{tropical Abel-Jacobi map}
$\mu\colon\Gamma\rightarrow\Jac(\Gamma)$ (with $\mu(v_0)=0$) such that the following diagram is commutative:
\[
\xymatrix@C=3em@R=3em{
\Gamma^{\ab}\ar[r]^{\sr} \ar[d] & H \ar[d]^{\varpi} \\
\Gamma\ar[r]^(.4){\mu} & \Jac(\Gamma).
}
\]
The $\Z$-module $\Div(\Gamma)$ of \emph{tropical divisors} on $\Gamma$ consists of finite formal
$\Z$-linear combinations of points on $\Gamma$ (including those on the edges) with the degree map
\[
\deg\colon\Div(\Gamma)\longrightarrow\Z
\]
given by the sum of the coefficients.
For $m\in\Z$, denote by $\Div^m(\Gamma)$ the set of divisors of degree $m$. 
The tropical Abel-Jacobi map $\mu$ determines an obvious map $\mu_m\colon\Div^m(\Gamma)\rightarrow\Jac(\Gamma)$.

The boundary map $\partial$ induces 
\[
\partial\colon C_1(\Gamma,\Z)\longrightarrow\Div^0(\Gamma),\qquad
\partial(e_j)=\source(e_j)-\target(e_j),\quad\text{for all }j\in J,
\]
and a commutative diagram
\[
\xymatrix@C=3em@R=3em{
C_1(\Gamma,\Z)\ar[r]^{\pi}\ar[d]_{\partial} & H\ar[d]^{\varpi} \\
\Div^0(\Gamma)\ar[r]^{\mu_0} & \Jac(\Gamma)
}
\]
with
\[
\varpi(\pi(e_j))=\mu_0(\partial e_j)=\mu(\source(e_j))-\mu(\target(e_j)),\quad\text{for all }j\in J.
\]
For each vertex $v_i$, let
\begin{align*}
\valency_+(v_i)&:=\#\{j\in J;\;\source(e_j)=v_i\}\\
\valency_-(v_i)&:=\#\{j\in J;\;\target(e_j)=v_i\}\\
\valency(v_i)&:=\valency_+(v_i)+\valency_-(v_i).
\end{align*}
Following Mikhalkin-Zharkov \cite{MikhalkinZharkov}, denote
\begin{align*}
K_+&:=\sum_{i\in I}(\valency_+(v_i)-1)v_i\in\Div^{g-1}(\Gamma)\\
K_-&:=\sum_{i\in I}(\valency_-(v_i)-1)v_i\in\Div^{g-1}(\Gamma)\\
K &:=K_++K_-=\sum_{i\in I}(\valency(v_i)-2)v_i\in\Div^{2g-2}(\Gamma).
\end{align*}
$K$ is a \emph{tropical canonical divisor}.
Since
\[
\varpi(\pi(e_j))=\mu_0(\partial e_j)=\mu(\source(e_j))-\mu(\target(e_j)),\quad\text{for all }j\in J,
\]
we easily see that
\[
\varpi(\pi(e(J)))=\mu_0(K_+-K_-).
\]
As in \cite{MikhalkinZharkov}, we see that
$[\Theta]:=\varpi(\Sk^{g-1}(\Vor(H,H_{\Z})))\subset\Jac(\Gamma)$ is the \emph{tropical theta divisor},
while the image $W_{g-1}\subset\Jac(\Gamma)$ of the obvious map
\[
\mu_{g-1}\colon\underbrace{\Gamma\times\cdots\times\Gamma}_\text{$(g-1)$-times}\rightarrow\Jac(\Gamma)
\] 
turns out to be
\[
W_{g-1}=\varpi\left(\Sk^{g-1}\left(\Vor\left(H,\pi\left(\frac{e(J)}{2}\right)+H_{\Z}\right)\right)\right),
\]
hence we have \emph{tropical Riemann's theorem} (cf.\ \cite[Cor.\ 8.6]{MikhalkinZharkov})
\[
W_{g-1}=[\Theta]+\frac{1}{2}\mu_0(K_+-K_-),
\]
which obviously contains the image of the tropical Abel-Jacobi map $\mu$.

However, it is more natural as in Alexeev \cite{Alexeev} to consider  $\Pic^{g-1}(\Gamma)$,
which is a principal homogeneous space under $\Pic^0(\Gamma)=\Jac(\Gamma)$.
The canonical theta divisor $\Theta_{g-1}\subset\Pic^{g-1}(\Gamma)$ is the image of the
set of effective divisors on $\Gamma$ of degree $g-1$.
\end{remark}

\section{Examples}

\begin{example}[Graphene crystal] \label{ex_graphene}

The graphene is the unique $2$-dimensional \emph{strongly isotropic} crystal (cf.\ Sunada \cite{Sunada}).

\[
\xymatrix@C=5em{
v_0 \ar@/^1.5pc/[r]|{e_1} \ar[r]|{e_2} & v_1 \ar@/^1.5pc/[l]|{e_3}
}
\qquad\qquad\qquad
\begin{array}{l}
\text{The elementary cycles:}\\
	\pm\left\{
	\begin{array}{l}
	\gamma_1:=e_1+e_3 \\
	\gamma_2:=e_2+e_3 \\
	\gamma_1-\gamma_2=e_1-e_2
	\end{array}
	\right\}.
\end{array}
\]
\begin{figure}[ht]
\centering
\includegraphics[trim=4cm 1cm 4cm 22cm, clip, width=.45\linewidth]{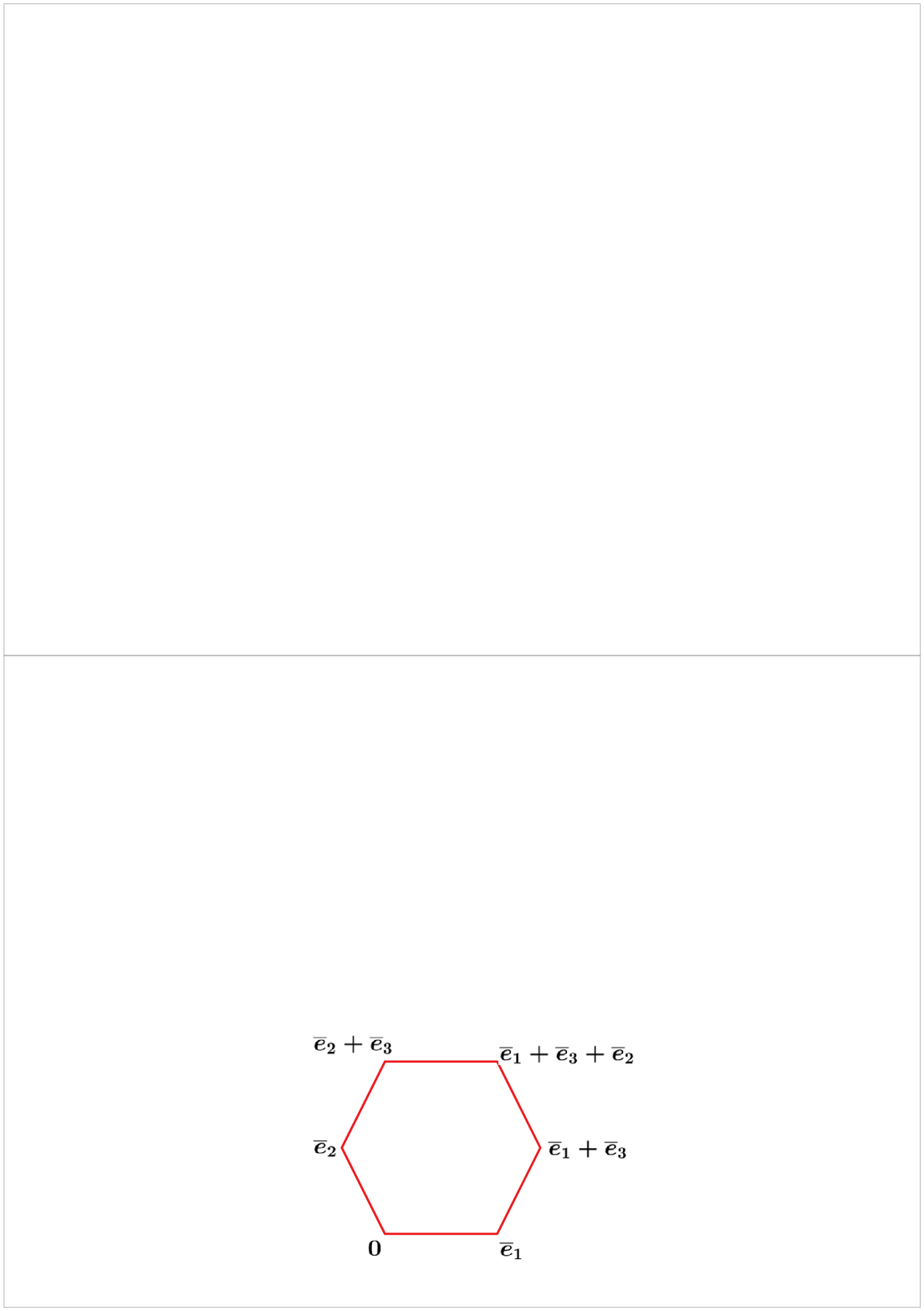}
\caption{Graphene}
\label{fig_graphene}
\end{figure}
For simplicity, denote $\ole_j:=\pi(e_j)$ for all $j$.
Then, $\delta v_0=-\delta v_1=e_1+e_2-e_3$ implies
\[
\ole_1+\ole_2=\ole_3,
\]
and
\[
 \left\{
 \begin{array}{l}
 \gamma_1=2\ole_1+\ole_2  \\[1ex]
 \gamma_2=\ole_1+2\ole_2  
 \end{array}
 \right.
 \qquad\text{hence}\qquad
 \left\{
 \begin{array}{l}
 {\displaystyle \ole_1=\frac{2\gamma_1-\gamma_2}{3}} \\[2ex]
 {\displaystyle \ole_2=\frac{-\gamma_1+2\gamma_2}{3}}.
 \end{array}
 \right.
\]
In view of
\[
(\gamma_1,\gamma_1)=(\gamma_2,\gamma_2)=2,\quad(\gamma_1,\gamma_2)=1,
\]
we have
\[
(\ole_1,\ole_1)=(\ole_2,\ole_2)=2/3,\quad (\ole_1,\ole_2)=-1/3
\]
and
\[
(\gamma_1,\ole_1)=(\gamma_2,\ole_2)=1,\quad(\gamma_1,\ole_2)=(\gamma_2,\ole_1)=0,
\]
hence $\{\gamma_1,\gamma_2\}$ and $\{\ole_1,\ole_2\}$ are mutually dual bases of
$H_{\Z}$ and $\pi(\Lambda)$, respectively.
The Voronoi cell $\pi(D_J)$ is a regular hexagon, and $(\Crystal(\Gamma))\cap\pi(D_J)$ is its circumference,
that is (cf.\ Figure \ref{fig_graphene}),
\begin{align*}
[0,\ole_1]&\cup[\ole_1,\ole_1+\ole_3]\cup[\ole_1+\ole_3,\ole_1+\ole_3+\ole_2]\\
&\cup[0,\ole_2]\cup[\ole_2,\ole_2+\ole_3]\cup[\ole_2+\ole_3,\ole_2+\ole_3+\ole_1].
\end{align*}
\end{example}

\begin{example}[Diamond crystal] \label{ex_diamond}

The diamond is one of the two $3$-dimensional \emph{strongly isotropic} crystals (cf.\ Sunada \cite{Sunada}).

\[
\xymatrix@C=5em{
v_0 \ar@/^2pc/[r]|{e_1} \ar@/^.7pc/[r]|{e_2} 
          \ar@/_.7pc/[r]|{e_3}  & v_1 \ar@/^2pc/[l]|{e_4}
}
\qquad\qquad\qquad
\begin{array}{l}
\text{The elementary cycles:}\\
	\pm\left\{
	\begin{array}{l}
	\gamma_1:=e_1+e_4\\
	\gamma_2:=e_2+e_4\\
	\gamma_3:=e_3+e_4\\
	\gamma_1-\gamma_2=e_1-e_2\\
	\gamma_2-\gamma_3=e_2-e_3\\
	\gamma_3-\gamma_1=-e_1+e_3
	\end{array}
	\right\}.
\end{array}
\]
For simplicity, denote $\ole_j:=\pi(e_j)$ for all $j$.
Then $\delta v_0=-\delta v_1=e_1+e_2+e_3-e_4$ implies
\[
\ole_4=\ole_1+\ole_2+\ole_3.
\]
Thus
\[
\left\{
\begin{array}{l}
\gamma_1=2\ole_1+\ole_2+\ole_3\\[1ex]
\gamma_2=\ole_1+2\ole_2+\ole_3\\[1ex]
\gamma_3=\ole_1+\ole_2+2\ole_3
\end{array}
\right.
\qquad\text{hence}\qquad
\left\{
\begin{array}{l}
{\displaystyle \ole_1=\frac{3\gamma_1-\gamma_2-\gamma_3}{4}}\\[2ex]
{\displaystyle \ole_2=\frac{-\gamma_1+3\gamma_2-\gamma_3}{4}}\\[2ex]
{\displaystyle \ole_3=\frac{-\gamma_1-\gamma_2+3\gamma_3}{4}}.
\end{array}
\right.
\]
In view of
\[
(\gamma_1,\gamma_1)=(\gamma_2,\gamma_2)=(\gamma_3,\gamma_3)=2,\quad
(\gamma_1,\gamma_2)=(\gamma_2,\gamma_3)=(\gamma_1,\gamma_3)=1,
\]
we get
\[
(\ole_1,\ole_1)=(\ole_2,\ole_2)=(\ole_3,\ole_3)=1/2,\quad
(\ole_1,\ole_2)=(\ole_2,\ole_3)=(\ole_1,\ole_3)=-1/4\\
\]
with $\{\gamma_1,\gamma_2,\gamma_3\}$ and $\{\ole_1,\ole_2,\ole_3\}$ being mutually dual bases
of $H_{\Z}$ and $\pi(\Lambda)$, respectively.

Let
\[
\left\{
\begin{array}{l}
u_1:=\ole_2+\ole_3\\
u_2:=\ole_1+\ole_3\\
u_3:=\ole_1+\ole_2.
\end{array}
\right.
\]
Then it is easy to see that $\{u_1,u_2,u_3\}$ is orthonormal and
\[
H_{\Z}=\Z\gamma_1+\Z\gamma_2+\Z\gamma_3\subset\Z u_1+\Z u_2+\Z u_3
\subset\pi(\Lambda)=\Z \ole_1+\Z\ole_2+\Z\ole_3
\]
with
\[
\left\{
\begin{array}{l}
\gamma_1=u_2+u_3\\
\gamma_2=u_1+u_3\\
\gamma_3=u_1+u_2
\end{array}
\right.
\]
and each lattice is a sublattice of index $2$ of the one to the right.
In particular, $H_{\Z}$ is the \emph{face centered cubic} (fcc) lattice, which is also known as $A_3$.

The Voronoi cell $\pi(D_J)$ is a rhombic dodecahedron (cf.\ Figure \ref{fig_diamondrhomb12}) with
\begin{figure}[ht]
\centering
\includegraphics[trim=1cm 1cm 1cm 21cm, clip, width=.7\linewidth]{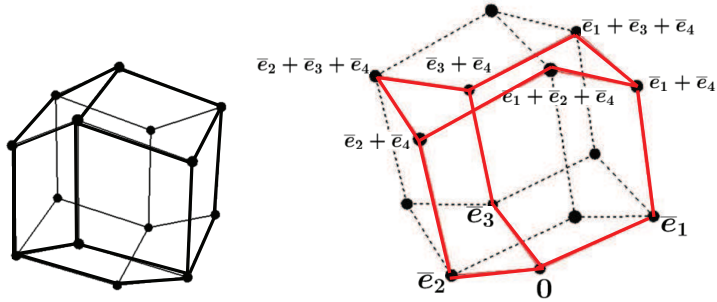}
\caption{Rhombic dodecahedron and diamond crystal on its surface}
\label{fig_diamondrhomb12} 
\end{figure}
\begin{eqnarray*}
\lefteqn{(\Crystal(\Gamma))\cap\pi(D_J)=}\\
&&[0,\ole_1]\cup[\ole_1,\ole_1+\ole_4]\cup[\ole_1+\ole_4,\ole_1+\ole_4+\ole_2]
 \cup[\ole_1+\ole_4,\ole_1+\ole_4+\ole_3]\\
&&\cup[0,\ole_2]\cup[\ole_2,\ole_2+\ole_4]\cup[\ole_2+\ole_4,\ole_2+\ole_4+\ole_1]
 \cup[\ole_2+\ole_4,\ole_2+\ole_4+\ole_3]\\
&&\cup[0,\ole_3]\cup[\ole_3,\ole_3+\ole_4]\cup[\ole_3+\ole_4,\ole_3+\ole_4+\ole_1]
 \cup[\ole_3+\ole_4,\ole_3+\ole_4+\ole_2].
\end{eqnarray*}
(cf.\ Figure \ref{fig_diamondrhomb12}).
\end{example}

\begin{example}[$K_4$ crystal] \label{ex_K4}

The $K_4$ crystal was shown by Sunada \cite{Sunada}
to possess the ``maximal symmetry'' and ``strong isotropy'' properties,
which it (together with its mirror image) shares in dimension three only with the diamond crystal.

\[
\xymatrix@R=3em@C=4em{
 & v_2 \ar[ldd]|{f_1} & \\
 & v_3 \ar[u]|{e_2}  \ar[rd]|{e_1}& \\
 v_0 \ar[ru]|{e_3} \ar[rr]|{f_2}& & v_1 \ar[luu]|{f_3}
}
\qquad\qquad
\begin{array}{l}
\text{The elementary cycles:}\\
	\pm\left\{
	\begin{array}{l}
	\gamma_1:=e_2+e_3+f_1\\
	\gamma_2:=-e_1-e_3+f_2\\
	\gamma_3:=e_1-e_2+f_3\\
	\gamma'_1:=-e_2-e_3+f_2+f_3=\gamma_2+\gamma_3\\
	\gamma'_2:=e_1+e_3+f_1+f_3=\gamma_1+\gamma_3\\
	\gamma'_3:=-e_1+e_2+f_1+f_2=\gamma_1+\gamma_2\\
	\gamma_0:=f_1+f_2+f_3=\gamma_1+\gamma_2+\gamma_3
	\end{array}
	\right\}.
\end{array}
\]
For simplicity, denote
$\ole_j:=\pi(e_j)$ and $\olf_j:=\pi(f_j)$ for all $j$.
Then
\begin{align*}
\delta v_0&=e_3-f_1+f_2\\
\delta v_1&=-e_1-f_2+f_3\\
\delta v_2&=-e_2+f_1-f_3\\
\delta v_3&=e_1+e_2-e_3
\end{align*}
implies
\[
\left\{
\begin{array}{l}
\ole_1=-\olf_2+\olf_3\\[1ex]
\ole_2=\olf_1-\olf_3\\[1ex]
\ole_3=\olf_1-\olf_2.
\end{array}
\right.
\]
Hence
\begin{align*}
&\left\{
\begin{array}{l}
\gamma_1=3\olf_1-\olf_2-\olf_3\\[1ex]
\gamma_2=-\olf_1+3\olf_2-\olf_3\\[1ex]
\gamma_3=-\olf_1-\olf_2+3\olf_3
\end{array}
\right.
&\left\{
\begin{array}{l}
\gamma_0=\olf_1+\olf_2+\olf_3\\[1ex]
\gamma'_1=2(-\olf_1+\olf_2+\olf_3)\\[1ex]
\gamma'_2=2(\olf_1-\olf_2+\olf_3)\\[1ex]
\gamma'_3=2(\olf_1+\olf_2-\olf_3).
\end{array}
\right.
\end{align*}
Consequently,
\begin{align*}
&\left\{
\begin{array}{l}
{\displaystyle \olf_1=\frac{2\gamma_1+\gamma_2+\gamma_3}{4}}\\[2ex]
{\displaystyle \olf_2=\frac{\gamma_1+2\gamma_2+\gamma_3}{4}}\\[2ex]
{\displaystyle \olf_3=\frac{\gamma_1+\gamma_2+2\gamma_3}{4}}
\end{array}
\right.
&\left\{
\begin{array}{l}
{\displaystyle \ole_1=\frac{-\gamma_2+\gamma_3}{4}}\\[2ex]
{\displaystyle \ole_2=\frac{\gamma_1-\gamma_3}{4}}\\[2ex]
{\displaystyle \ole_3=\frac{\gamma_1-\gamma_2}{4}}.
\end{array}
\right.
\end{align*}
In view of
\[
(\gamma_1,\gamma_1)=(\gamma_2,\gamma_2)=(\gamma_3,\gamma_3)=3,\quad
(\gamma_1,\gamma_2)=(\gamma_2,\gamma_3)=(\gamma_1,\gamma_3)=-1,
\]
we have
\[
(\olf_1,\olf_1)=(\olf_2,\olf_2)=(\olf_3,\olf_3)=1/2,\quad
(\olf_1,\olf_2)=(\olf_2,\olf_3)=(\olf_1,\olf_3)=1/4
\]
with $\{\gamma_1,\gamma_2,\gamma_3\}$ and $\{\olf_1,\olf_2,\olf_3\}$ being mutually dual bases
of $H_{\Z}$ and $\pi(\Lambda)$, respectively.

Let
\[
\left\{
\begin{array}{l}
u_1:=-\olf_1+\olf_2+\olf_3\\
u_2:=\olf_1-\olf_2+\olf_3\\
u_3:=\olf_1+\olf_2-\olf_3.
\end{array}
\right.
\]
Then it is easy to see that $\{u_1,u_2,u_3\}$ is orthonormal and
\[
H_{\Z}=\Z\gamma_1\oplus\Z\gamma_2\oplus\Z\gamma_3\subset\Z u_1\oplus\Z u_2\oplus\Z u_3\subset
\pi(\Lambda)=\Z\olf_1\oplus\Z\olf_2\oplus\Z\olf_3
\]
with
\[
\left\{
\begin{array}{l}
\gamma_1=-u_1+u_2+u_3\\[1ex]
\gamma_2=u_1-u_2+u_3\\[1ex]
\gamma_3=u_1+u_2-u_3
\end{array}
\right.
\]
and each lattice is a sublattice of index $4$ of the one to the right.
In particular, $H_{\Z}$ is the \emph{body centered cubic} (bcc) lattice, which is also known as $A_3^{\ast}$.
The Voronoi cell $\pi(D_J)$ is a truncated octahedron (also known as Kelvin polytope) 
(cf.\ Figure \ref{fig_k4kelvin})
\begin{figure}[ht]
\centering
\includegraphics[trim=1cm 1cm 1cm 21cm, clip, width=.8\linewidth]{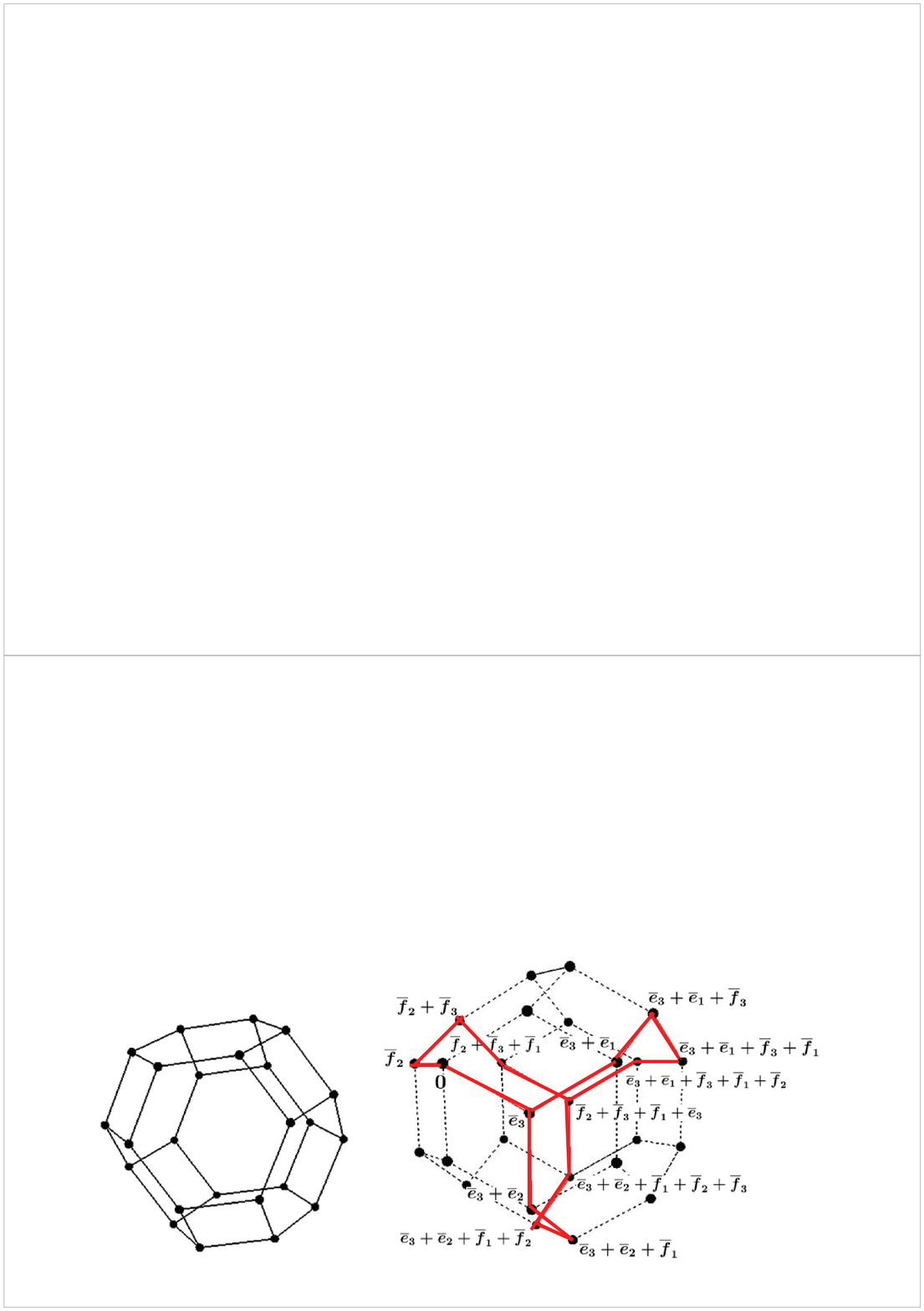}
\caption{Truncated octahedron (also known as Kelvin polytope) and $K_4$ crystal on its surface}
\label{fig_k4kelvin} 
\end{figure}
such that
\begin{eqnarray*}
\lefteqn{(\Crystal(\Gamma))\cap\pi(D_J)=} \\
&&[0,\ole_3]\cup[\ole_3,\ole_3+\ole_1]\cup[\ole_3+\ole_1,\ole_3+\ole_1+\olf_3]\\
&&\cup[\ole_3+\ole_1+\olf_3,\ole_3+\ole_1+\olf_3+\olf_1]
 \cup[\ole_3+\ole_1+\olf_3+\olf_1,\ole_3+\ole_1+\olf_3+\olf_1+\olf_2]\\
&&\cup[\ole_3,\ole_3+\ole_2]\cup[\ole_3+\ole_2,\ole_3+\ole_2+\olf_1]
 \cup[\ole_3+\ole_2+\olf_1,\ole_3+\ole_2+\olf_1+\olf_2]\\
&&\cup[\ole_3+\ole_2+\olf_1+\olf_2,\ole_3+\ole_2+\olf_1+\olf_2+\olf_3] \\
&&\cup[0,\olf_2]\cup[\olf_2,\olf_2+\olf_3]\cup[\olf_2+\olf_3,\olf_2+\olf_3+\olf_1]
 \cup[\olf_2+\olf_3+\olf_1,\olf_2+\olf_3+\olf_1+\ole_3]\\
&&\cup[\olf_2+\olf_3+\olf_1+\ole_3,\olf_2+\olf_3+\olf_1+\ole_3+\ole_1]
 \cup[\olf_2+\olf_3+\olf_1+\ole_3,\olf_2+\olf_3+\olf_1+\ole_3+\ole_2].
\end{eqnarray*}
(cf.\ Figure \ref{fig_k4kelvin})
\end{example}

\begin{example}[Lonsdaleite (also known as Hexagonal diamond) crystal] \label{ex_lonsdaleite}
\[
\xymatrix@C=6em@R=5em{
v_3 \ar@/^1.5pc/[d]|{n_3} & v_2 \ar@/^1.5pc/[d]|{m_3} \ar[l]|{l_2} \\
v_0 \ar@/^1.5pc/[u]|{n_1} \ar[u]|{n_2} \ar[r]|{l_1} & v_1 \ar[u]|{m_2} \ar@/^1.5pc/[u]|{m_1}
}
\]

We have
\[
H_{\Z}=\Z(l_1-m_3+l_2+n_3)\oplus\Z(m_1+m_3)\oplus\Z(m_2+m_3)\oplus\Z(n_1+n_3)\oplus\Z(n_2+n_3).
\]
Let
\[
L:=\Z(m_1+m_3-n_1-n_3)\oplus\Z(m_2+m_3-n_2-n_3)\subset H_{\Z}
\]
and consider the free abelian covering
\[
\widetilde{\Gamma}:=\Gamma^{\ab}/L\longrightarrow\Gamma.
\]
The associated standard realization of $\widetilde{\Gamma}$ in the $3$-dimensional Euclidean space
\[
E':=(H_{\Z}/L)\otimes_{\Z}\R
\]
turns out to be the Lonsdaleite crystal.
Since
\[
\left\{
\begin{array}{l}
\delta v_0=n_1+n_2-n_3+l_1\\[1ex]
\delta v_1=m_1+m_2-m_3-l_1\\[1ex]
\delta v_2=-m_1-m_2+m_3+l_2\\[1ex]
\delta v_3=-n_1-n_2+n_3-l_2,
\end{array}
\right.
\]
we see that
\[
C_1(\Gamma,\R)=:C\supset E'
=\left\{
\begin{array}{l}
n_1+n_2-n_3+l_1,m_1+m_2-m_3-l_1,\\
-m_1-m_2+m_3+l_2,-n_1-n_2+n_3-l_2,\\
m_1+m_3-n_1-n_3,m_2+m_3-n_2-n_3
\end{array}
\right\}^{\perp}
\]
with the orthogonal projection $\pi'\colon C\rightarrow E'$.
For simplicity denote
\[
l'_j:=\pi'(l_j),m'_j:=\pi'(m_j),n'_j:=\pi'(n_j),\quad\text{for all }j.
\]
Then obviously
\begin{align*}
l'_1=m'_1+m'_2-m'_3&=-n'_1-n'_2+n'_3\\[1ex]
l'_2=m'_1+m'_2-m'_3&=-n'_1-n'_2+n'_3\\[1ex]
m'_1+m'_2&=n'_1+n'_3\\
m'_2+m'_3&=n'_2+n'_3.
\end{align*}
Let
\begin{align*}
q_1&:=-m'_2+n'_3=-n'_2+m'_3 \\
q_2&:=-m'_1+n'_3=-n'_1+m'_3 \\
q_3&:=q_1+q_2-m'_3=-q_1-q_2+n'_3.
\end{align*}
Then easy computation shows that
\begin{align*}
m'_1&=q_1+q_3 \\
m'_2&=q_2+q_3 \\
m'_3&=q_1+q_2-q_3 \\
n'_1&=q_1-q_3\\
n'_2&=q_2-q_3\\
n'_3&=q_1+q_2+q_3 \\
l'_1&=3q_3 \\
l'_2&=3q_3. 
\end{align*}
With $\Lambda:=C_1(\Gamma,\Z)$, we see easily that
\[
\Lambda\cap E'=\Z\gamma_1\oplus\Z\gamma_2\oplus\Z\gamma_3,
\]
where
\begin{align*}
\gamma_1&:=m_1+m_3+n_1+n_3 \\
\gamma_2&:=m_2+m_3+n_2+n_3 \\
\gamma_3&:=m_1+m_2-m_3-n_1-n_2+n_3+3l_1+3l_2. 
\end{align*}
In $E'$ we have a sequence of lattices
\[
\Lambda\cap E'\subset\pi'(H_{\Z})\subset\pi'(\Lambda)
\]
with
\begin{align*}
\pi'(\Lambda)&=\Z q_1\oplus\Z q_2\oplus\Z q_3\\
\pi'(H_{\Z})&=\Z(2q_1+2q_2)\oplus\Z(q_1+2q_2)\oplus\Z(4q_3)\\
\Lambda\cap E'&=\Z\gamma_1\oplus\Z\gamma_2\oplus\Z\gamma_3.
\end{align*}
Since
\[
\left\{
\begin{array}{l}
\gamma_1=4q_1+2q_2\\[1ex]
\gamma_2=2q_1+4q_2\\[1ex]
\gamma_3=24q_3
\end{array}
\right.
\qquad\text{hence}\qquad
\left\{
\begin{array}{l}
{\displaystyle q_1=\frac{2\gamma_1-\gamma_2}{6}}\\[2ex]
{\displaystyle q_2=\frac{-\gamma_1+2\gamma_2}{6}}\\[2ex]
{\displaystyle q_3=\frac{\gamma_3}{24}}
\end{array}
\right.
\]
with
\[
(\gamma_1,\gamma_1)=(\gamma_2,\gamma_2)=4,\quad
(\gamma_3,\gamma_3)=24,\quad
(\gamma_1,\gamma_2)=2,\quad
(\gamma_1,\gamma_3)=(\gamma_2,\gamma_3)=0,
\]
we see that
\[
(q_1,q_1)=(q_2,q_2)=1/3,\quad (q_1,q_2)=-1/6,\quad (q_3,q_3)=1/24,\quad (q_1,q_3)=(q_2,q_3)=0
\]
and that $\{\gamma_1,\gamma_2,\gamma_3\}$ and $\{q_1,q_2,q_3\}$ 
are mutually dual bases of $\Lambda\cap E'$ and $\pi'(\Lambda)$, respectively.

The Lonsdaleite crystal, which is obtained as the standard realization 
$\pi'(\Crystal(\Gamma))$ of a non-maximal abelian covering, also turns out not to intrude 
the interiors of the top-dimensional Voronoi cells in
\[
\Vor(E',q_1+q_2+3q_3+\pi'(H_{\Z})).
\]
The Voronoi cell 
\[
V(q_1+q_2+3q_3)\in\Vor(E',q_1+q_2+3q_3+\pi'(H_{\Z}))
\]
is a regular hexagonal cylinder (cf.\ Figure \ref{fig_lonsdaleite}) with vertices
\[
\begin{array}{l}
-q_3,q_1-q_3,2q_1+q_2-q_3,2q_1+2q_2-q_3,q_1+2q_2-q_3,q_2-q_3,\\
7q_3,q_1+7q_3,2q_1+q_2+7q_3,2q_1+2q_2+7q_3,q_1+2q_2+7q_3,q_2+7q_3.
\end{array}
\]
\begin{figure}[ht]
\centering
\includegraphics[trim=1cm 1cm 1cm 21cm, clip, width=.8\linewidth]{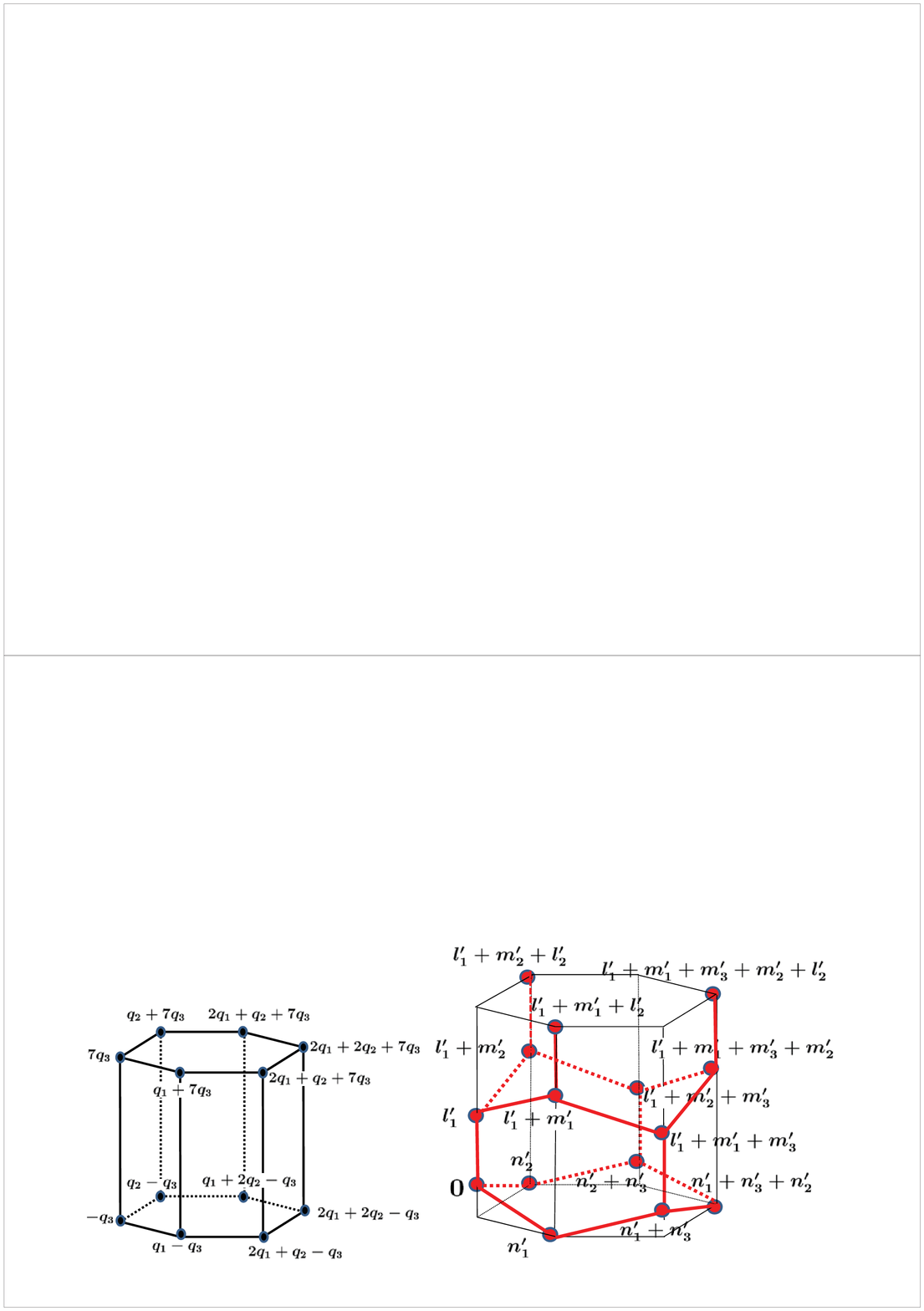}
\caption{Regular hexagonal cylinder and Lonsdaleite crystal on its surface}
\label{fig_lonsdaleite} 
\end{figure}
$V(q_1+q_2+3q_3)\cap\pi'(\Crystal(\Gamma))$ consists of the following
(cf.\ Figure \ref{fig_lonsdaleite}):

\begin{itemize}
\item the hexagonal closed curve joining
\[
\left\{
	\begin{array}{lcl}
	0 & & \\
	n'_1&=&q_1-q_3\\
	n'_1+n'_3&=&2q_1+q_2\\
	n'_1+n'_3+n'_2&=&2q_1+2q_2-q_3\\
	n'_2+n'_3&=&q_1+2q_2\\
	n'_2&=&q_2-q_3\\
	0 &&
	\end{array}
	\right.
\]
in this order,

\item the hexagonal closed curve joining
\[
\left\{
	\begin{array}{lcl}
	l'_1&=&3q_3\\
	l'_1+m'_1&=&q_1+4q_3\\
	l'_1+m'_1+m'_3&=&2q_1+q_2+3q_3\\
	l'_1+m'_1+m'_3+m'_2&=&2q_1+2q_2+4q_3\\
	l'_1+m'_2+m'_3&=&q_1+2q_2+3q_3\\
	l'_1+m'_2&=&q_2+4q_3\\
	l'_1&=&3q_3
	\end{array}
	\right.
\]
in this order,

\item the line segment joining
\[
\left\{
 	\begin{array}{lcl}
 	0&&\\
 	l'_1&=&3q_3
 	\end{array}
 	\right.
\]

\item the line segment joining
\[
\left\{
 	\begin{array}{lcl}
 	n'_1+n'_3&=&2q_1+q_2\\
 	n'_1+n'_3+l'_1&=&2q_1+q_2+3q_3
 	\end{array}
 	\right.
\]

\item the line segment joining
\[
\left\{
 	\begin{array}{lcl}
 	n'_2+n'_3&=&q_1+2q_2\\
 	n'_2+n'_3+l'_1&=&q_1+2q_2+3q_3
 	\end{array}
 	\right.
\]

\item the line segment joining
\[
\left\{
 	\begin{array}{lcl}
 	l'_1+m'_1&=&q_1+4q_3\\
 	l'_1+m'_1+l'_2&=&q_1+7q_3
 	\end{array}
 	\right.
\]

\item the line segment joining
\[
\left\{
 	\begin{array}{lcl}
 	l'_1+m'_1+m'_3+m'_2&=&2q_1+2q_2+4q_3\\
 	l'_1+m'_1+m'_3+m'_2+l'_2&=&2q_1+2q_2+7q_3
 	\end{array}
 	\right.
\]

\item the line segment joining
\[ 
\left\{
 	\begin{array}{lcl}
 	l'_1+m'_2&=&q_2+4q_3\\
 	l'_1+m'_2+l'_2&=&q_2+7q_3
 	\end{array}
 	\right.
\]
\end{itemize}
In this case, we have a bijection
\[
\Gamma=\widetilde{\Gamma}/(H_{\Z}/L)\stackrel{\sim}{\longrightarrow}
\pi'(\Crystal(\Gamma))/\pi'(H_{\Z}).
\]

\end{example}


\bigskip

\begin{flushleft}
\textsc{Professor Emeritus}\\
\textsc{Tohoku University}\\[2ex]
\textit{E-mail address}: \texttt{odatadao@math.tohoku.ac.jp}
\end{flushleft}


\begin{thebibliography}{O-S}
\bibitem[A]{Alexeev} V. Alexeev, Compactified Jacobians and Torelli map,
Publ.\ Res. Inst. Math. Sci., Kyoto Univ.\ 40 (2004), 1241--1265.

\bibitem[B]{Bondy} J. A. Bondy, Basic graph theory: paths and circuits,
Chap.\ 1 of \textit{Handbook of Combinatorics} (R. I. Graham, M. Gr\"otschel and L. Lov\'asz, eds.),
Elsevier Science B. V. and MIT Press, 1995.

\bibitem[E]{Erdahl} R. M. Erdahl, Zonotope, dicings, and Voronoi's conjecture on parallelohedra,
Europ. J. Combinatorics 20 (1999), 527--549.

\bibitem[E-R]{ErdahlRyshkov} R. M. Erdahl and S. S. Ryskov, On lattice dicing, 
Europ. J. Combinatorics 15 (1994), 459--481.

\bibitem[F]{Frank} A. Frank, Connectivity and network flows,
Chap.\ 2 of \textit{Handbook of Combinatorics} (R. I. Graham, M. Gr\"otschel and L. Lov\'asz, eds.),
Elsevier Science B. V. and MIT Press, 1995.

\bibitem[K-S1]{KotaniSunada1} M. Kotani and T. Sunada, Standard realizations of crystal lattices 
via harmonic maps, Trans. Amer. Math. Soc. 353 (2000), 1--20.

\bibitem[K-S2]{KotaniSunada2} M. Kotani and T. Sunada, Large deviation and the tangent cone
at infinity of a crystal lattice, Math. Z. 254 (2006), 837--870.

\bibitem[Mc1]{McMullen1} P. McMullen, On zonotopes, Trans.\ Amer.\ Math.\ Soc.\
159 (1971), 91--110.

\bibitem[Mc2]{McMullen2} P. McMullen, Space tiling zonotopes, Mathematika 22 (1975), 202--211.

\bibitem[Mc3]{McMullen3} P. McMullen, Convex bodies which tile space by translation,
Mathematika 27 (1980), 113--121.

\bibitem[M-Z]{MikhalkinZharkov} G. Mikhalkin and I. Zharkov, Tropical curves, their Jacobians
and theta functions, arXiv:math/0612267v2 [math.AG] 30 Nov 2007.

\bibitem[O]{OdaCutAndProject} T. Oda, Convex polyhedral tilings hidden in crystals and
quasicrystals, submitted.

\bibitem[O-S]{OdaSeshadri} T. Oda and C. S. Seshadri, Compactifications of the 
generalized Jacobian variety, Trans. Amer. Math. Soc. 253 (1979), 1--90.

\bibitem[R]{Robbins} H. E. Robbins, A theorem on graphs with an application to a problem of
traffic control, Amer.\ Math.\ Monthly 46 (1939), 281--283.

\bibitem[S1]{Sunada} T. Sunada, Crystals that nature might miss creating,
Notices of Amer. Math. Soc. 55 (2008), 208--215. 

\bibitem[S2]{Sunada-book} T. Sunada, Topological Crystallography 
---In View of Discrete Geometric Analysis---, to appear.

\bibitem[V]{Venkov} B. A. Venkov, On a class of Euclidean polytopes (in Russian),
Vestnik Leningrad Univ., Ser.\ Mat.\ Fiz.\ Him.\ 9 (1954), 11--31.
\end{thebibliography}
\end{document}